\documentclass[12pt,reqno]{amsart}
\usepackage{amssymb, latexsym, color}

\newcommand{\beqa}{\begin{eqnarray*}}
\newcommand{\eeqa}{\end{eqnarray*}}
\newcommand{\beqn}{\begin{eqnarray}}
\newcommand{\eeqn}{\end{eqnarray}}

\newcommand{\bQ}{\mathbb Q}

\newcommand{\D}{\mathbb D}
\newcommand{\G}{\mathbb G}
\newcommand{\R}{\mathbb R}

\newcommand{\N}{\mathbb N}

\newcommand{\mcA}{\mathcal A}

\newcommand{\mcB}{\mathcal B}

\newcommand{\mcD}{\mathcal D}
\newcommand{\mcE}{\mathcal E}
\newcommand{\mcF}{\mathcal F}
\newcommand{\mcP}{\mathcal P}

\newcommand{\mcS}{\mathcal S}

\newcommand{\mfA}{\mathfrak A}
\newcommand{\mfB}{\mathfrak B}

\newcommand{\mfD}{\mathfrak D}

\newcommand{\mfU}{\mathfrak U}

\newcounter{cnt1}
\newcounter{cnt2}
\newcounter{cnt3}
\newcommand{\blr}{\begin{list}{$($\roman{cnt1}$)$}
 {\usecounter{cnt1} \setlength{\topsep}{0pt}
 \setlength{\itemsep}{0pt}}}
\newcommand{\bla}{\begin{list}{$($\alph{cnt2}$)$}
 {\usecounter{cnt2} \setlength{\topsep}{0pt}
 \setlength{\itemsep}{0pt}}}
\newcommand{\bln}{\begin{list}{$($\arabic{cnt3}$)$}
 {\usecounter{cnt3} \setlength{\topsep}{0pt}
 \setlength{\itemsep}{0pt}}}
\newcommand{\el}{\end{list}}
\newtheorem{thm}{Theorem}[section]
\newtheorem{lem}[thm]{Lemma}
\newtheorem{cor}[thm]{Corollary}
\newtheorem{ex}[thm]{Example}

\newtheorem{Def}[thm]{Definition}
\newtheorem{prop}[thm]{Proposition}
\newtheorem{rem}[thm]{Remark}
\newcommand{\Rem}{\begin{rem} \rm}
\newcommand{\bdfn}{\begin{Def} \rm}
\newcommand{\edfn}{\end{Def}}

\newcommand{\ba}{\begin{array}}
\newcommand{\ea}{\end{array}}

\begin{document}
			\begin{center}\large{{\bf{HK-Sobolev spaces  $W{S^{k,p}}$  and Bessel Potential}}} 
		\vspace{0.05cm}
		
		Bipan Hazarika$^{a}$ \footnote{Corresponding author} and  Hemanta Kalita$^a$
		
				\vspace{0.05cm}
				$^a$Department of Mathematics, Gauhati University, Guwahati 781014,\\ Assam, India\\	
	Email :  bh\_rgu@yahoo.co.in; bh\_gu@gauhati.ac.in; 
 hemanta30kalita@gmail.com\\
	
\end{center}
	\title{}
	\author{}
	\thanks{\today} 
	\begin{abstract} 
Our goal in this article is to construct HK-Sobolev  spaces on $\R^\infty$ which contains Sobolev spaces as dense embedding. We show that weakly convergent sequences in Sobolev spaces are strongly convergent in HK-Sobolev  spaces.  Also, we obtain that    the Sobolev space through Bessel potential is densely contained in HK-Sobolev  spaces. Finally we find sufficient  conditions for the solvability of the divergence equation  $\nabla\cdot F= f,$ when $f$ is an element of the subspace $K{S^p}[\R_I^n]$  of the HK-Sobolev  space $WS^{k,p}[\R_I^n] $  with the help  of the Fourier transformation.
		AMS Subject Classification No(2020): 26A39, 46B25, 46E35, 46E39, 46F25.
	\end{abstract}
	\maketitle
	
	
	\pagestyle{myheadings}
	\markboth{\rightline {\scriptsize  HK}}
	{\leftline{\scriptsize  HK-Sobolev  space\dots }}
	
	\maketitle


\section{Introduction and Preliminaries} 
One of the most important problems of mathematical physics in the 20th century was  to find the solution to Dirichlet and Neuman problems for Laplace equation  (see for instance \cite{WM}). This problem  attracted  famous scientists of that period, namely  Hilbert, Courant, Weyl and many more. Russian Mathematician Sergei Sobolev in 1930 overcame the main difficulty of this problem and introduced a functional space  called Sobolev space, given by functions in $L^p[\R^n]$ whose distributional derivatives of order upto to $k $ exist and are  in $L^p[\R^n].$ Today there are many information about  Sobolev spaces $W^{k,p}[\R^n]$, where $p>1 $ and $k=0,1,2, \dots$,  (see \cite{JH,GL,WM,MV}).  
In  \cite{GZ}, Gill and Zachary  built the $KS^p$-spaces (Kuelbs-Steadman spaces). See also \cite{VS}. The Kuelbs-Steadman spaces $K{S^p}[\R^n]$ were introduced to cover Feynman path integral formulation, an alternative approach to Quantum Mechanics. This spaces have been useful in this approach since they contain Henstock-Kurzweil integrable functions, that are fundamental in order to prove the convergent of highly oscillatory integrals that appear in Feynman approach. Also the space $K{S^p}[\R^n] $ is the completion of $L^p[\R^n]$ for $1 \leq p \leq \infty.$ Even more interesting is that $L^p[\R^n]$ is continuous densely contained in $K{S^p}[\R^n]$ and this spaces contains the spaces of distribution as dense subset.  Gill and Myers   \cite{GM} discussed about   a new theory of Lebesgue measure on $\R^\infty;$ the construction of which is virtually the same as the development of Lebesgue measure on $\R^n.$ This theory can be useful in formulating a new class of spaces which will provide a Banach spaces  structure for Henstock-Kurzweil (HK) integrable functions. For details of Henstock-Kurzweil integral (in short HK integral) the readers can see  \cite{PrezBecerra,PrezBecerra1,SB,AB,RA,BH,RH, SnchezPerales,SnchezPeralesTorres,VST,VST1, CS,TY}.  Motivated by the concept of \cite{GZ}, the fact $L^p \subset K{S^p}$ as continuous dense embedding and parallel approach of $\R^n $ and $\R^\infty$ (see \cite{GM, YY}), we introduce our function spaces $W{S^{k,p}}[\R_I^n]$, $W{S^{k,p}}[\R_I^\infty]$, which we will call as HK-Sobolev  spaces.

\begin{Def}\cite{GM,YY}
Let $\mathcal{B}[\mathbb{R}^n]$ be the Borel $\sigma$-algebra for $\mathbb{R}^n$,  $I=[-\frac{1}{2},\frac{1}{2}]$ and $I_n=\prod_{i=n+1}^\infty I$. For $\mfA\in\mathcal{B}[\mathbb{R}^n]$  the set $ \mfA_n = \mfA \times I_n$ is called $n^{th}$ order box set in $ \mathbb{R}^\infty$. We define 
   \begin{enumerate}
   \item $ \mfA_n \cup \mfB_n =(\mfA\cup \mfB) \times I_n;$
   \item $ \mfA_n \cap \mfB_n = (\mfA \cap \mfB ) \times I_n; $
   \item $ \mfB_{n}^c = \mfB^c \times I_n.$
   \end{enumerate}
   \end{Def}
 \begin{Def}
  \cite{GM,YY} Define $\mathbb{ R}_{I}^{n} =\mathbb{ R}^n \times I_n.$  We denote $ \mcB[\mathbb{R}_{I}^{n}]$ to be the Borel $\sigma$-algebra for $ \mathbb{R}_{I}^{n},$ where the topology for $ \mathbb{R}_{I}^{n} $ is defined via the class of open sets $ \mfD_n = \{ \mfU \times I_n : \mfU $ is open in $ \mathbb{R}^n\}.$ For any $ \mfA \in \mcB[\mathbb{R}^n], $ we define $\lambda_\infty(\mfA_n) $ on $ \mathbb{R}_{I}^{n} $ by product measure $ \lambda_\infty(\mfA_n) = \lambda_n(\mfA) \times \Pi_{i= n +1 }^{\infty} \lambda_{1}(I) = \lambda_n(\mfA),$ where $\lambda_n$ is Lebesgue measure on $\mathbb{R}^n$.
  \end{Def}
   \begin{thm} \cite{GZ,GM}
 $ \lambda_{\infty}(.) $ is a measure on $ \mcB[\mathbb{R}_{I}^{n}],$ which  is equivalent to $n$-dimensional Lebesgue measure on $ \mathbb{R}^n.$   
 \end{thm}
\begin{cor}\cite{GZ,GM}
  The measure $ \lambda_{\infty}(.) $ is both translationally and rotationally invariant on $ (\mathbb{R}_{I}^{n}, \mcB[\mathbb{R}_{I}^{n}]) $ for each $ n \in \mathbb{N}.$
\end{cor}
   We can construct a theory on $ \mathbb{R}_{I}^{n} $ that completely parallels that on $ \mathbb{R}^n. $     Since $ \mathbb{R}_{I}^{n} \subset \mathbb{R}_{I}^{n+1},$ we have an increasing sequence, so we define $ \widehat{\mathbb{R}}_{I}^{\infty} = \lim\limits_{ n \to \infty} \mathbb{R}_{I}^{n} = \bigcup\limits_{n=1}^{\infty} \mathbb{R}_{I}^{n}.$ In \cite {GM} it is shown that the measure $ \lambda_{\infty}(.) $ can be extended to $ \mathbb{R}^\infty.$  
     Let $ x =(x_1,x_2,\dots) \in \mathbb{R}_{I}^{\infty}.$ Also let $ I_n = \Pi_{k=n+1}^{\infty}[\frac{-1}{2}, \frac{1}{2}]$ and let $ h_n(\widehat{x})= \chi_{I_n}(\widehat{x}),$ where $\widehat{x} = (x_i)_{i=n+1}^{\infty}.$ 
   Recalling  $\mathbb{R}_{I}^{\infty} $ is the closure of $\widehat{\mathbb{R}}_{I}^{\infty} $ in the induced topology from $ \mathbb{R}^{\infty}.$ From our construction, it is clear that a set of the form $ \mfA= \mfA_n \times (\Pi_{k=n+1}^{\infty}\mathbb{R}) $ is not in $\widehat{\mathbb{R}}_{I}^{\infty}$ for any $n.$ So, $ \widehat{\mathbb{R}}_{I}^{\infty} \neq \mathbb{R}^{\infty}$. The natural topology for $\mathbb{R}_{I}^{\infty}$ is that induced as a closed subspace of $\mathbb{R}^{\infty}.$ Thus if $x= (x_n),~y=(y_n) $ are sequences in $\mathbb{R}_{I}^{\infty},$   a metric $d$ on $\mathbb{R}_{I}^{\infty},$ is defined as  
   $$d(x, y) = \sum\limits_{n=1}^{\infty} \frac{1}{2^n} \frac{|x_n - y_n|}{1+|x_n -y_n|}.$$
   \begin{thm}
  $\mathbb{ R}_{I}^{\infty}= \mathbb{R}^{\infty}$ as sets but not as topological spaces.
   \end{thm}
    We call $\mathbb{R}_{I}^{\infty} $ the essentially bounded version of $\mathbb{R}^{\infty}.$ There are certain pathologies of $ \mathbb{R}^{\infty}$ that are preserved to $\mathbb{R}_{I}^{\infty}$, for example, if $\mcA_i$ has measure $ 1+ \epsilon $ for all $i$ then $\lambda_{\infty}(\mcA)= \Pi_{i=1}^{\infty}\lambda(\mcA_i)= \infty.$ On the other hand, if each $ \mcA_i$ has measure $ 1-\epsilon,$ then $\lambda_{\infty}(\mcA)= \Pi_{i=1}^{\infty}\lambda(\mcA_i)=0$. Thus the class of sets $ \mcA \in \mathcal{B}[\mathbb{R}_{I}^{\infty}] $ for which $ 0 < \lambda_{\infty}(\mcA) < \infty$ is relatively small. It follows that the sets of measure zero need not be small nor sets of infinite measure be large.
    \subsection{Measurable function}
  We discuss about measurable function on $\R_I^\infty$ as follows:\\
   Let $ x =(x_1,x_2,\dots) \in \R_I^\infty,$  $ I_n = \Pi_{k=n+1}^{\infty}[\frac{-1}{2}, \frac{1}{2}] $ and let $ h_n(\widehat{x})= \chi_{I_n}(\widehat{x}),$ where $\widehat{x} = (x_i)_{i=n+1}^{\infty}.$
     \begin{Def}
     \cite{GM} Let $M^n $ be represented the class of measurable functions on $\R^n.$ If $ x \in \R_I^\infty$ and $f^n \in M^n,$ Let $\overline{x}= (x_i)_{i=1}^{n} $ and define an essentially tame measurable function of order $n $ (or $ e_n-$tame ) on $\R_I^\infty$ by $$f(x)= f^n(\overline{x}) \otimes h_n(\widehat{x}).$$ We let $M_I^n = \{f(x)~: f(x) = f^n( \overline{x}) \otimes h_n(\widehat{x}),~x \in \R_I^\infty \}$ be the class of all $e_n-$ tame functions.
     \end{Def}
       \begin{Def}
     A function $f: \R_I^\infty \to \R$ is said to be measurable and we write $ f \in M_I$, if there is a sequence $\{f_n \in M_I^n \} $ of $e_n-$ tame functions, such that $$\lim\limits_{ n \to \infty} f_n(x) \to f(x)~\mu_\infty-(a.e.).$$
     \end{Def}
  \subsection{ $L^1$-Theory in $\R_I^\infty$}
  Let $L^1[\R_I^n] $ be the class of integrable functions on $\R_I^n.$ Since $\R_I^n \subset \R_I^{n+1}$ we define $L^1[\widehat{\R}_I^\infty]= \bigcup\limits_{n=1}^{\infty}L^1[\R_I^n].$ We say that a measurable function $ f \in L^1[\R_I^\infty]$ if there exists a Cauchy sequence $\{f_n\}\subset L^1[\widehat{\R}_I^\infty]$ with $f_n \in L^1[\R_I^n]$ and $\lim\limits_{ n \to \infty}f_n(x)=f(x)~, \lambda_\infty $-(a.e.).\\
   The integral of $f \in L^1[\R_I^\infty]$ is defined by $$\int_{\R_I^\infty}f(x)d\lambda_\infty(x)= \lim\limits_{ n \to \infty}\int_{\R_I^\infty}f_n(x)d\lambda_\infty,$$ where $\{f_n\} \subset L^1[\R_I^\infty]$ is any Cauchy-sequence converges to $f(x)$-a.e.
  \begin{thm}\label{th18}
  $L^1[\widehat{\R}_I^\infty]= L^1[\R_I^\infty].$
  \end{thm} 
    Let $C_c[\R_I^n]$ be the class of continuous function on $\R_I^n$ which vanish outside compact sets. We say that a measurable function $f \in C_c[\R_I^\infty],$ if there exists a Cauchy-sequence$\{f_n\} \subset \bigcup\limits_{n=1}^{\infty}C_c[\R_I^n]=C_c[\widehat{\R}_I^\infty]$ such that $\lim\limits_{n \to \infty}||f_n -f||_\infty=0.$  We define $C_c[\R_I^\infty]$, the continuous functions that vanish at $\infty$, and  $C_0^\infty[\R_I^\infty]$ the bounded functions,  in similar way. 
    \begin{thm}
    $C_0^\infty[\widehat{\R}_I^\infty]= C_0^\infty[\R_I^\infty].$
    \end{thm}
    \begin{thm}
    $C_c[\R_I^\infty] $ is dense in $L^1[\R_I^\infty].$
    \end{thm}
    \begin{thm}
    $C_0^\infty[\R_I^\infty] $ is dense in $L^1[\R_I^\infty].$
    \end{thm}
    \begin{proof}
    Since $C_0^\infty[\R_I^n] \subset L^1[\R_I^n]$ as dense. So $\cup C_0^\infty[\R_I^n] \subset \cup L^1[\R_I^n]$ as dense. This gives $C_0^\infty[\cup \R_I^n] \subset L^1[\cup \R_I^n]$ as dense. Now $\lim\limits_{n \to \infty}C_0^\infty[\cup \R_I^n] \subset \lim\limits_{n \to \infty}  L^1[\cup \R_I^n].$ This implies $ C_0^\infty[\lim\limits_{n \to \infty}\cup \R_I^n] \subset L^1[\lim\limits_{n \to \infty}\cup \R_I^n].$ So, $C_0^\infty[\widehat{\R}_I^\infty] \subset L^1[\widehat{\R}_I^\infty]=L^1[\R_I^\infty]$ as dense.
    \end{proof}
    \begin{rem}
    In a similar fashion we can define the $L_{loc}^{1}[\R_I^\infty].$
    \end{rem}
    \subsection{$L^p$-Theory in $\R_I^\infty$:} The $L^p$ spaces are function spaces defined using a natural generalization of the $p$-norm 
    for finite-dimensional vector spaces. They are sometimes called Lebesgue spaces. $L^p$ spaces form an 
    important class of Banach spaces in functional analysis and  topological vector spaces. Because of     their key role in the mathematical analysis of measure and probability spaces. Lebesgue spaces are also  used  in the theoretical discussion of problems in physics, statistics, finance, engineering, and
    other disciplines.
      We  now construct the spaces $L^p[\R_I^\infty],~1<p<\infty$, using the same approach that led to $L^1[\R_I^\infty].$ Since $L^p[\R_I^n] \subset L^p[\R_I^{n+1}],~ $ we define $L^p[\widehat{\R}_I^\infty] = \cup_{n=1}^{\infty}L^p[\R_I^n].$ We say that a measurable function $ f \in L^p[\R_I^\infty],$ if there is a Cauchy-sequence $\{f_n\} \subset L^p[\widehat{\R}_I^\infty]$ such that $\lim\limits_{ n \to \infty}||f_n -f||_p = 0.$\\
    Similar to Theorem \ref{th18}, we have that functions in $L^p[\widehat{\R}_I^\infty]$ differ from functions in its closure $L^p[\R_I^\infty], $ by sets of measure zero.
    \begin{thm}
     $L^p[\widehat{\R}_I^\infty]= L^p[\R_I^\infty].$
    \end{thm}
    \begin{Def}\label{def114}
    If $f \in L^p[\R_I^\infty],$ we define the integral of $f$ by \begin{equation}\label{eq111}
    \int_{\R_I^\infty}f(x)d\lambda_\infty(x)= \lim\limits_{n \to \infty}\int_{\R_I^\infty}f_n(x)d\lambda_\infty(x).
    \end{equation}
    \end{Def}
    \begin{thm}
    If $f \in L^p[\R_I^\infty],$ then above integral exists.
    \end{thm}
    \begin{proof}
    The proof follows from the fact that the sequence in the Definition \ref{def114} is of $L^p$-Cauchy.
    \end{proof}
    If $f$ is a measurable function on $\R_I^\infty$ and $1<p<\infty,$ we define $$||f||_p=\left[ \int_{\R_I^\infty}|f|^pd\lambda_\infty(x)\right]^\frac{1}{p}.$$ 
    \begin{thm}
    If $f \in L^p[\R_I^\infty],$ then the integral of (\ref{eq111}) exists and all theorems that are true for $f \in L^p[\R_I^n],$ also hold for $f \in L^p[\R_I^\infty].$
    \end{thm}
   
    \begin{thm}
    $C_c[\R_I^\infty] $ is dense in $L^p[\R_I^\infty].$
    \end{thm}
    \begin{thm}
    $C_0^\infty[\R_I^\infty] $ is dense in $L^p[\R_I^\infty].$
    \end{thm}
    \subsection{ Theory of $K{S^p}[\R_I^\infty]$}
   Gill and Zachary in \cite{GZ} introduced a new Banach space of non-absolute integrable funcions. Kuelbs lemma (see \cite{KB}) was the main tool of their work with approach of \cite{VS}, they found a kind of Banach spaces  called Kuelbs-Steadman spaces that contain  Henstock-Kurzweil integrable functions. Construction of $K{S^p}[\R^n]$ was discussed in detailed in their work. We adopt their approach with virtual concept of Lebesgue measure in $\R^\infty.$ We now construct the spaces $K{S^p}[\R_I^\infty],~ 1 \leq p \leq \infty,$ using the same approach that led to the construction of $L^1[\R_I^\infty].$ Since $K{S^p}[\R_I^n] \subset K{S^p}[\R_I^{n+1}],$ we define $K{S^p}[\widehat{\R}_I^\infty] = \bigcup\limits_{n=1}^{\infty}K{S^p}[\R_I^n].$ 
   \begin{Def}
       We say that a measurable function $ f \in K{S^p}[\R_I^\infty],$ for $ 1 \leq p \leq \infty,$ if there is a Cauchy sequence $\{f_n\} \subset K{S^p}[\widehat{\mathbb{R}}_{I}^{\infty}] $ with $ f_n \in K{S^p}[\R_I^n] $ and $\lim\limits_{n \to \infty}f_n(x) = f(x)~ \mu_\infty$-a.e.
       \end{Def}
       The functions in $K{S^p}[\widehat{\R}_I^\infty] $ differ from functions in its closure $K{S^p}[\R_I^\infty],$ by sets of measure zero.
       \begin{thm}
       $K{S^p}[\widehat{\R}_I^\infty]= K{S^p}[\R_I^\infty] $
       \end{thm}
       \begin{Def}\label{Def17}
     If $ f \in K{S^p}[\R_I^\infty],$ we define the integral of $ f$ by $$ \int_{\R_I^\infty} f(x)d\lambda_\infty(x) = \lim\limits_{n \to \infty}\int_{\R_I^n} f_n(x)d\lambda_\infty(x),$$ where $ f_n \in K{S^p}[\R_I^n]$ is any  Cauchy sequence converging to $f.$
     \end{Def}
     \begin{thm}
   If $ f \in K{S^p}[\R_I^\infty]$, then the integral of $ f$ defined in Definition \ref{Def17} exists and is unique for every $f \in K{S^p}[\R_I^\infty]$.
\end{thm}
\begin{proof} Since the family of functions $\{f_n\}$ is Cauchy, it follows that if the integral exists, it is unique.  To prove the existence, follow the standard argument. First assume that $f(x) \ge 0.$  In this case, the sequence can always be chosen  increasing, so that the integral exists.  The general case now follows by the standard decomposition. 
\end{proof}
\subsection{\bf{Construction:}} Fix $n\in\mathbb{N}$ and let $\widehat{\bQ}_{I}^\infty = \lim\limits_{ n \to \infty}\bQ_{I}^n= \bigcup\limits_{k=1}^{\infty}\bQ_{I}^k,$ where $\bQ_{I}^n$ is the set $\{x\in \R_I^n :$  the coordinates of $x$ are rational$\}$. Since this is a countable dense set in $\R_I^n,$ we can arrange it as $\bQ_{I}^n=\{x_1,x_2,...\}$. For each $k$ and $i,$ let $\mathcal{B}_k(x_i)$ be a closed cube in $\mathbb{R}^n$ centered at $x_i$ with sides parallel to the coordinate axes and edge $e_k= \frac{1}{2^{k} \sqrt n}.$ Now choose the natural order which maps $\N \times \N$ bijectively to $\N,$ and let $\{\mathcal{B}_k:~k \in \N\}$ be the resulting set of (all) closed cubes $$\{\mathcal{B}_k(x_i)|~(k,i) \in \N\times\N \}$$  centered at a point in $\bQ_{I}^n.$ Let $\zeta_k(x) $ be the characteristic function of $\mathcal{B}_k,$ so that $\zeta_k(x) \in L^p[\R_I^\infty] \cap L^\infty[\R_I^\infty]$ for $ 1 \leq p < \infty.$     Define $F_k(.)$ on $L^1[\R_I^\infty]$ by $$F_k(f)= \int_{\R_I^\infty}\zeta_k(x)f(x)d \lambda_\infty(x).$$ Since each $\mathcal{B}_k$ is a cube with sides parallel to the coordinate axes, $F_k(.)$ is well defined for all HK-integrable functions. Also it is a bounded linear functional on $L^p[\R_I^\infty]$ for $ 1 \leq p \leq \infty.$  Fix $\tau_k>0 $ such that $ \sum_{k=1}^{\infty} \tau_k=1$ and defined an inner product $(\cdot)$ on $L^1[\R_I^\infty]$ by $$(f,g)= \sum_{r=1}^{\infty}\tau_k\left[ \int_{\R_I^\infty}\zeta_k(x) f(x)d \lambda_\infty(x)\right]\left[\int_{\R_I^\infty}\zeta_k(y)g(y) d \lambda_\infty(y)\right]^c.$$ The completion of $L^1[\R_I^\infty]$ in  the  inner product is the space  $KS^2[\R_I^\infty].$ To see directly that $K{S^2}[\R_I^\infty] $ contains the HK-integrable functions, observe 

\begin{align*}
||f||_{K{S^2}}^{2} &= \sum_{k=1}^{\infty}\tau_k\left|\int_{\R_I^\infty}\zeta_k(x)f(x) d \lambda_\infty(x)\right|^2 \\&\leq \sup_{k}\left|\int_{\mathcal{B}_k}f(x) d \lambda_\infty(x)\right|^2< \infty.
\end{align*}
So, $ f \in K{S^2}[\R_I^\infty].$

\begin{thm}\label{th21}
   The space $ K{S^2}[\R_I^\infty] $ contains $ L^p[\R_I^\infty]$ for $1 \leq p \leq \infty $ as dense subspace.
   \end{thm} 
   \begin{proof}
Since $K{S^2}[\R_I^n] $ contains $ L^p[\R_I^n]$ for each $p,~1\leq p \leq \infty$ as dense subspace and  $K{S^2}[\R_I^\infty] $ is the closure of $\bigcup\limits_{n=1}^{\infty} K{S^2}[\R_I^n],$ it follows that $K{S^2}[\R_I^\infty]$ contains the closure of $\bigcup\limits_{n=1}^{\infty}L^p[\R_I^n] $, but this closure is  $L^p[\R_I^\infty]$.
    \end{proof}
  Before proceeding more, we define a norm on $L^p[\R_I^\infty]$ as $$||f||_{K{S^p}[\R_I^\infty]} =\left\{\begin{array}{c}\left(\sum\limits_{k=1}^{\infty}\tau_{k}\left|\int_{\R_{I}^{\infty}}\zeta_{k}(x)f(x)d\lambda_\infty(x)\right|^p\right)^{\frac{1}{p}}, \mbox{~for~} 1\leq p<\infty;\\
  \sup\limits_{k\geq 1}\left|\int_{\R_{I}^{\infty}}\zeta_{k}(x)f(x)d\lambda_\infty(x)\right|, \mbox{~for~} p=\infty \end{array}\right.$$ where $x=(x_1,x_2,x_3,\dots) \in \R_I^\infty.$ It is easy to see that $||f||_{K{S^p}[\R_I^\infty]}$ defines a norm on $ L^p[\R_I^\infty].$ The completion of $L^p[\R_I^\infty] $ with this norm is $K{S^p}[\R_I^\infty]$.
    \begin{thm}
    For each $q$ with $1 \leq q \leq \infty,~ K{S^p}[\R_I^\infty] \supset L^q[\R_I^\infty] $ as a dense continuous embedding.
    \end{thm}
 \begin{proof}
 This is easily deduced from the fact that  $ K{S^p}[\R_I^n] \supset L^q[\R_I^n] $ as a dense continuous embedding for each $q,~1 \leq q \leq \infty.$ 
 \end{proof}
\begin{thm}\label{th25}
  For $ 1 \leq p \leq \infty,$ we have 
  \begin{enumerate} 
  \item  If $ f_n \to f$ weakly in $ L^p[\R_I^\infty]$ then $ f_n \to f $ strongly in $ K{S^p}[\R_I^\infty].$
  \item  If $ 1<p<\infty $, then $ K{S^p}[\R_I^\infty]$ is uniformly convex.
  \item  If $ 1<p<\infty $ and $\frac{1}{p}+ \frac{1}{q}=1,$ then the dual space of $ K{S^p}[\R_I^\infty] $ is $ K{S^q}[\R_I^\infty].$
  \item  $ K{S^{\infty}}[\R_I^\infty] \subset K{S^p}[\R_I^\infty] $ for $ 1\leq p< \infty.$
  \item $ C_c[\R_I^\infty] $ is dense in $ K{S^p}[\R_I^\infty].$
  \end{enumerate}  
  \end{thm}
\section*{Our goal of the article:} In this article we first observe that the weak differentiability  in $L^p[\R_I^n]$ is the strong differentiability in $K{S^p}[\R_I^n]$ and this is also true when we replace $\R_I^\infty $ by $\R_I^n.$ 
Secondly, our purpose is to claim that weakly differentiability in $ W^{k,p}[\R_I^n]$ is strongly differentiability in $W{S^{k,p}}[\R_I^n].$ Finally we show that weakly convergent sequences in $W^{k,p}[\R_I^n]$ and $W^{k,p}[\R_I^\infty]$ are strongly convergent in $W{S^{k,p}}[\R_I^n]$ and $W{S^{k,p}}[\R_I^\infty],$ respectively. As an application, in the last section we found sufficient conditions to solve the equation $\nabla\cdot F= f,$ for $f$  an element of the subspace $K{S^p}[\R_I^n]$ of the HK-Sobolev  space  $W{S^{k,p}}[\R_I^n] $, with the help of the Fourier transformation.  
\section{Meaning of $\D^k f(x)$ when $x \in \R_I^\infty$}
The test functions $\mcD[\R_I^n]$  on $\R_I^n$ are similar as test functions on $\R^n,$ so ignore  the detailed of the test functions on $\R_I^n.$

We denote   test functions on $\R_I^\infty$ as $\mcD[\R_I^\infty],$ to construct this spaces  we use  the same approach that led to $L^1[\R_I^\infty]$ in subsection $1.2.$ Since $\mcD[\R_I^n] \subset \mcD[\R_I^{n+1}],~ $ we define $\mcD[\widehat{\R}_I^\infty] = \cup_{n=1}^{\infty}\mcD[\R_I^n].$
 \begin{Def}
  We say that a measurable function $ f \in \mcD[\R_I^\infty] $ if and only if there exists a sequence of functions $\{f_m \} \subset \mcD[\widehat{\R_I^\infty}] = \bigcup\limits_{n=1}^{\infty}\mcD[\R_I^n]$  and a compact set $K \subset \R_I^\infty,$ which contains the support of $f -f_m$ for all $m,$ and $\mathbb{D}^\alpha f_m \to \mathbb{D}^\alpha f$ uniformly on $K,$ for every multi index $ \alpha \in \N_0^\infty.$   We call the topology of $\mcD[\R_I^\infty] $ as  the compact sequential limit topology.
  \end{Def}
  \begin{thm}\label{th}
  For each $p,~1 \leq p \leq \infty,$ then test function $\mcD[\R_I^n] \subset K{S^p}[\R_I^n] $ as a continuous embedding.
  \end{thm}
  \begin{proof}
  Proof is similar as the proof of the Theorem 3.47 of \cite{GZ}.
  \end{proof}
  \begin{rem}
  The Theorem (\ref{th}) is hold for $\R_I^\infty.$
  \end{rem}
The mollifiers are used in distribution theory to create sequences of smooth functions that approximate non smooth functions via convolution. Sergie Sobolev \cite{Sobolev} in the year 1938 used mollifier functions in his work Sobolev embedding theorem. Modern approach of mollifier was introduced by Kurt Otto Friedrichs \cite{Friedrichs} in the year 1944.
\begin{Def}
(Friedrichs's Definition) Mollifier  identified  the convolution operator as 
\begin{eqnarray*}
\phi_\epsilon(f)(x)= \int_{\R^n}\varphi_\epsilon(x-y)f(y)dy
\end{eqnarray*}
where $\varphi_\epsilon(x)=\epsilon^{-n}\varphi(\frac{x}{\epsilon})$ and $\varphi$ is a smooth function satisfying 
 \begin{enumerate}
 \item $\varphi(x)\geq 0 $ for all $ x \in \R^n.$
 \item $\varphi(x)=\mu(|x|)$ for some infinitely differentiable function $\mu: \R^{+} \to \R.$
 \end{enumerate}
\end{Def}
To construct mollifier in $\R_I^\infty,$ for each $\epsilon > 0,$ let $\varphi_\epsilon \in C_0^\infty[\R_I^\infty]$ be given with the property $$\varphi_\epsilon \geq 0,~\mbox{supp}(\varphi_\epsilon) \subset \{ x \in \R_I^\infty~: |x| \leq \epsilon \},~\int \varphi_\epsilon =1 $$ such functions  can be constructed (see page 32 \cite{S}),
 for example, by taking an appropiate multiple of $$\varphi_\epsilon(x) =\left \{\begin{array}{c} \exp(|x|^2 -\epsilon^{2})^{-1} \mbox{~,~} |x| < \epsilon ;\\
  0 \mbox{~,~} |x| \geq \epsilon \end{array}\right.$$
  Let $ f \in L^1[\G],$ where $\G $ is open in $\R_I^\infty.$ Suppose that the support of $f $ satisfies $\mbox{supp}(f) \subset \subset \G$ (compact support),  then the distance from $\mbox{supp}(f) $ to $\partial \G $ is a positive number $\Delta.$ We extend $ f $ as zero on complement of $\G$ and also we denote the extension in $L^1[\R_I^\infty]$ by $f.$ Define for each $\epsilon$ the mollifier:

 \begin{equation}
  f_\epsilon(x)= \int_{\R_I^\infty}f(x-y)\varphi_\epsilon(y)d\lambda_\infty,~x \in \R_I^\infty.
  \end{equation}
   From now on we consider functions $f \in K{S^p}$ so, $f=0$ almost everywhere. We obtain the following
   lemma  
  \begin{lem}
  \begin{enumerate}
  \item For each $\epsilon >0,~supp(f_\epsilon) \subset supp(f)+ \{y~: |y| \leq \epsilon \} $ and $f_\epsilon \in C^\infty[\R_I^\infty].$
  \item If $ f \in C_0[\G],~~$ then $f_\epsilon \to f $ uniformly on $\G.$ If $f \in K{S^p}[\G],~1 \leq p < \infty$ then $||f_\epsilon||_{K{S^p}[\G]} \leq ||f||_{KS^p[\G]}$ and $f_\epsilon \to f $ in $K{S^p}[\G].$
  \end{enumerate}
  \end{lem}
  \begin{proof}
  For (1), the proof is similar to that of \cite[Lemma 1.1]{S}.\\
  For (2), use the fact that $L^p[\G]$ is dense as continuous embedding on $K{S^p}[\G]$ and follow the proof of  \cite[Lemma 1.2]{S}.
  \end{proof}
  \begin{thm}
  $C_0^\infty[\G] $ is a dense subset of $K{S^2}[\G] $ and $K{S^p}[\G]$.
  \end{thm}
  \begin{proof}
  Since $C_0^\infty[\G] $ is dense in $L^2[\G]$ and $L^p[\G]$, it follows that  $C_0^\infty[\G] $ is dense in $K{S^2}[\G] $ and $K{S^p}[\G]$. It follows the result. 
  \end{proof}
  \begin{Def}
  A distribution on $\G$ is a conjugate linear functional on $C_0^\infty[\G]$, that is $C_0^\infty[\G]^{*} $ is the linear space of distributions on $\G.$
  \end{Def}
  \begin{ex}
   The space $L_{loc}^{1}[\G] = \bigcap \{L^1[K]~~: K \subset \subset \G \}$ of locally integrable functions on $\G $ can be identified with a subspace of distributions on $\G.$ That is , $ f \in L_{loc}^{1}[\G]$ is assigned the distribution $T_f \in C_0^\infty[\G]^{*} $ defined by $$T_f(\varphi)= \int_{\G} f\varphi^c,~~\varphi \in C_0^\infty[\G]$$ where the HK integral over the support of $\varphi$ is used.
  \end{ex}
  \begin{rem} If 
  $ T: L_{loc}^{1}[\G] \to C_{0}^{\infty}[\G]^{*} $ is an injection, then the functions in $K{S^2}[\G]$ will be identified as a subspace of $\mcD^{*}[\G].$ 
  \end{rem}
 Let $\alpha= ( \alpha_1, \alpha_2, ..,)$ be multi-index of non negative integers with $|\alpha|= \sum\limits_{k=1}^{\infty} \alpha_k.$ We define the operators $\D_n^\alpha $ and $\D_{\alpha, n}$ by $$\D_n^\alpha= \Pi_{k=1}^{n} \frac{\partial^{\alpha_k}}{\partial x^{\alpha_k}}~\D_{\alpha, n}= \Pi_{k=1}^{n}\left(\frac{1}{2 \pi i}\frac{1}{\partial x_k}\right)^{\alpha_k},$$ respectively. 
 
  \begin{Def}
  \begin{enumerate}
  \item We say that a sequence of functions $\{f_m\} \subset C^\infty[\R_I^\infty]$ converges to a function $ f \in C^\infty[\R_I^\infty]$ if and only if for all multi-indices $\alpha,~~\D^\alpha f \in C[\R_I^\infty]$ and for $ x \in \R_I^\infty $  for all $ n \in \N$, such that  $$\lim\limits_{m \to \infty} \sup [ \sup_{\alpha} \sup_{||x|| \leq \N}|\D^\alpha f(x) - \D^\alpha f_m(x)|]=0.$$

  \end{enumerate}
  \end{Def}
\subsection{\bf{Observation}}:
  We say that a function $ f \in C^\infty[\R_I^\infty]$ if and only if there exists a sequence of functions $\{f_m\} \subset C^\infty[\widehat{\R_I^\infty}]= \bigcup\limits_{n=1}^{\infty}C^\infty[\R_I^n]$ such that for all $ x \in \R_I^\infty$ and $ n \in \N $,   $$\lim\limits_{m \to \infty} \sup [ \sup_{\alpha} \sup_{||x|| \leq \N}|\D^\alpha f(x) - \D^\alpha f_m(x)|]=0.$$
 
  From the above we can say the set of all continuous linear functionals $ T \in \mcD^{*}[\R_I^\infty] $  is called the space of distributions on $\R_I^\infty.$ A family of distributions $\{T_i\} \subset \mcD^{*}[\R_I^\infty] $ is said to converge to $T \in \mcD^{*}[\R_I^\infty]$ if for every $\varphi \in \mcD[\R_I^\infty],$ the numbers $T_i(\varphi) $ converge to $T(\varphi).$\\
   We define derivatives of distributions in such a way that it agrees with the usual notion of derivative in those distributions which arise from continuously differentiable functions. We define $\partial^{\alpha}:\mcD^{*}[\R_I^\infty] \to \mcD^{*}[\R_I^\infty]$ as $\partial ^{\alpha}(T_f)= T_{\D^\alpha f},~~|\alpha| \leq m,~~f \in C^m[\R_I^\infty].$ By integration by parts we obtain  $$T_{\D^\alpha f}(y)= (-1)^{|\alpha|} T_f( \D^\alpha \varphi),~~\varphi \in C_0^\infty[\R_I^\infty]$$ and this identity suggest the following definition:\\
   The $\alpha^{th}$ partial derivative of the distribution $T $ is the distribution $\partial^{\alpha}T $ defined by $$\partial^\alpha T (\varphi) = (-1)^{|\alpha|} T( \D^\alpha \varphi),~~\varphi \in C_0^\infty[\R_I^\infty].$$ Since $\D^\alpha \in L( C_0^\infty[\R_I^\infty], C_0^\infty[\R_I^\infty),$  it follows that $\partial^\alpha T $ is linear. Every distribution has derivatives of all orders and so is every function. For distribution theory one can see \cite{MA,DS,FM} and references therein. 
   \begin{ex} It is clear that the derivatives $\partial^{\alpha} $  and $\D^\alpha $ are compatible with identifications of $C^\infty[\R_I^\infty] $ in $\mcD^{*}[\R_I^\infty].$ For example:
   \begin{enumerate}
   \item If $ f \in C^1[\R_I^\infty]$ then $$\partial f(\varphi)= - f(\D \varphi)=- \int f (\D \varphi^c)= \int (\D f) \varphi^c= \D f(\varphi) $$ where the equality follows by integration by parts. In particular, if $ f(x)= H(x),$ where $H$ is the Heaveside function on $\R_I^\infty, $ $$H(x) =\left \{\begin{array}{c} 1 \mbox{~for~} x_i \geq 0;\\
  0 \mbox{~for~} x_i <0,~i \in \N\end{array}\right.$$ for $x=(x_1,x_2,x_3,\dots) \in \mathbb{R}_{I}^{\infty},$
  then
  \begin{align*}
  \int_{\R_I^\infty} \D H(x) \varphi(x) d\lambda_\infty(x) &= \int_{\R_I^\infty} H(x) \D \varphi(x) d \lambda_\infty(x) \\&= \varphi(0) \\&= \int_{\R_I^\infty} \partial_{\R_I^\infty}(x) \varphi(x) d \lambda_\infty(x).
  \end{align*}
    That is, in the generalized sense of distributions, $\D H(x)= \partial(x)$ the Dirac delta function on $\R_I^\infty$
  \item Let $ f: \R_I^\infty \to K $ be satisfy $ f_{| {\R_I^\infty}^{-} }\in C^\infty(- \infty, 0]$ 
  and $ f_{| {\R_I^\infty}^{+}} \in C^\infty[0, \infty)$ and denote the jump in the various derivatives at $0$ by $$ \sigma_m(f) = \D^m f(0^{+}) - \D^m f(0^{-}),~~m \geq 0.$$ Then we obtain $$\partial f (\varphi)= \D f(\varphi) + \sigma_{0}(f) \partial(\varphi),~~ \varphi \in C_0^\infty[\R_I^\infty].$$ That is $ \partial f = \D f + \sigma_{0}(f) \delta,$ we can compute derivatives of higher order as : \\
  $ \partial^{2} f = \D^2 f + \sigma_1(f) \delta + \sigma_0(f) \partial \delta\\
  \partial^3 f = \D^3 f + \sigma_2 (f) \delta + \sigma_1(f) \partial \delta + \sigma_0(f) \partial^2 \delta$\\
  eg $\partial(H\cdot\sin)= H\cdot\cos\\
  \partial(H\cdot\cos)= -H\cdot\sin + \delta. $ So, $H\cdot\sin $ is a  generalized solution of the ODE $(\delta^2 +1)y= \delta.$
   \end{enumerate}
\end{ex}

     \begin{Def}
   If $\alpha$ is a multi-index and $u,v \in L_{loc}^{1}[\R_I^\infty],$ we say that $v$ is the $\alpha^{th}$ weak (or distributional) partial derivative of $u $ and write $ \D^\alpha u= v $ provided that $$ \int_{\R_I^\infty}u(\D^\alpha \varphi) d \lambda_\infty= (-1)^{|\alpha|} \int_{\R_I^\infty} \varphi v d \lambda_\infty $$ for all functions $ \varphi \in C_c^\infty[\R_I^\infty].$ Thus $v$ is in the dual space $\mcD^{*}[\R_I^\infty] $ of $\mcD[\R_I^\infty].$
   \end{Def}
    If $ u \in L_{loc}^1[\R_I^\infty] $ and $\varphi \in \mcD[\R_I^\infty]$ then we can define $T_u(\cdot) $ by $$T_u(\varphi)= \int_{\R_I^\infty} u \varphi d \lambda_\infty.$$ This is a linear functional on $\mcD[\R_I^\infty]$. If $\{\varphi_n\} \subset \mcD[\R_I^\infty]$ and $\varphi_n \to \varphi $ in $\mcD[\R_I^\infty],$ with the support of $\varphi_n - \varphi$ contained in a compact set $ K \subset \R_I^\infty$,  then we have
    \begin{align*}
    |T_u(\varphi_n) - T_u(\varphi)| &= \left| \int_{\R_I^\infty} u(x)[ \varphi_n(x)-\varphi(x)] d \lambda_\infty(x)\right| \\&\leq \sup_{ x \in K}| \varphi_n(x)- \varphi(x)| \int_{\R_I^\infty}|u(x)|d \lambda_\infty(x).
    \end{align*}
     By uniform convergence on $K,$ we see that $T $ is continuous, so  $T \in \mcD^{*}[\R_I^\infty].$ We assume $$||\varphi||= \sup_{x \in \R_I^\infty}\{| \D^\alpha \varphi(x)|~: \alpha \in \N_0^\infty,~ |\alpha| \leq N \}.$$

     \begin{thm}
      Let $\mcD^{*}[\R_I^\infty] $ be the dual space of $\mcD[\R_I^\infty].$
     \begin{enumerate}
     \item Every differentiable operator $D^\alpha,~~\alpha \in \N_0^\infty $ defines a bounded linear operator on $\mcD[\R_I^\infty].$
     \item If $T \in \mcD^{*}[\R_I^\infty]$ and $\alpha \in \N_0^\infty,$ then $D^\alpha T \in \mcD^{*}[\R_I^\infty]$ and $$ (\D^\alpha T)(\varphi)= (-1)^{|\alpha|} T(\D^\alpha \varphi),~~\varphi \in \mcD[\R_I^\infty].$$
     \item If $|T(\varphi)| \leq c||\varphi||_{N} $ for all $ \varphi \in \mcD[K],$ for some compact set $K \subset \R_I^\infty,$ then $|(\D^\alpha T)(\varphi)| \leq c||\varphi||_{N+|\varphi|} $ and $\D^\alpha \D^\beta T= \D^\beta \D^\alpha T. $
     \item If $ g = \D^\alpha f $ exists as a classical derivative and $ g \in L_{loc}^1[\R_I^\infty],$ then $T_g \in \mcD^{*}[\R_I^\infty]$ and $$(-1)^{|\alpha|} \int_{\R_I^\infty} f(x) (\D^\alpha \varphi) d \lambda_\infty(x)= \int_{\R_I^\infty} g(x) \varphi(x) d \lambda_\infty(x)$$ for all $ \varphi \in \mcD[\R_I^\infty].$
     \item If $ f \in C^\infty[\R_I^\infty]$ and $T \in \mcD^{*}[\R_I^\infty]$ then $ f T \in \mcD^{*}[\R_I^\infty],$ with $ f T(\varphi) = T(f \varphi) $ for all $ \varphi \in \mcD[\R_I^\infty] $ and $\D^\alpha(fT) = \sum\limits_{\beta \leq \alpha} C_{\alpha \beta}(\D^{\alpha - \beta} f) (\D^\beta T) .$
     \end{enumerate}
     \end{thm}
     \begin{proof}
      The proofs are similar to those of $\R^n$.
      \end{proof}
       The weak and strong derivative for $L^p[\R_I^n]$ can be defined like the weak and strong derivative for $L^p[\R^n]$.  For  theory of the weak derivative and strong derivative for $L^p[\R^n]$ we follow the  definition 29.15 of  \cite{Driver}.
       \begin{thm}
      Strong differentiable implies weak differentiable in $L^p[\R_I^n].$
      \end{thm}
      \begin{proof}
       The proof is similar as $(4) \implies (2) $ of the theorem 29.18 of \cite{Driver} those of $L^p[\R^n].$ 
      \end{proof}
      We  state the weak and strong derivative for $L^p[\R_I^\infty]$ as:
      \begin{Def}
       Let $v \in \R_I^\infty$ and $f \in L^p[\R_I^\infty]~(f \in L_{loc}^{1}[\R_I^\infty])$, then $ \partial_{v}^wf $ is said to exists weakly in $L^p[\R_I^\infty]( L_{loc}^{1}[\R_I^\infty])$ if there exists a function $ g \in L^p[\R_I^\infty]( g \in L_{loc}^{1}[\R_I^\infty])$ such that $$< f, \partial_v \varphi>= -< g, \varphi>,~\forall \varphi \in C_c^\infty[\R_I^\infty].$$
  In this case $\partial_v^w f =g$.
      \end{Def}
      
      \begin{Def}
      \begin{enumerate}
      \item For $ v \in \R_I^\infty$, $h \in \R-\{0\}$ and a function $f: \R_I^\infty \to C,$ let $$\partial_{v^h} f(x) = \frac{f(x+hv)-f(x)}{h}$$ for those $ x \in \R_I^\infty$ such that $ x+ h v \in \R_I^\infty$. When $v $ is one of the standard basis elements $e_i$, for $ 1 \leq i \leq d,$ we will write $\partial_{i}^{h} f(x) $ rather than $ \partial_{{e_i^h}} f(x)$.
      \item  Let $ v \in \R_I^\infty $ and $ f \in L^p[\R_I^\infty],$ then it is said that $\partial_v^s f $  exists strongly in $L^p[\R_I^\infty]$, if $ \lim\limits_{h \to 0} \partial_{v}^{h} f $ exists in $L^p[\R_I^\infty]$. In this case $\partial_v^s f=\lim\limits_{h \to 0} \partial_{v^{h}} f$.
      \end{enumerate}
     \end{Def}
      
      Now we give the strong and the weakly differentiability for functions in $K{S^p}[\R_I^\infty]$.  However, to understand this, we need little work on $K{S^p}[\R_I^n]$. 
      \begin{Def}
      Let $v \in \R^n$ and $f \in K{S^p}[\R_I^n]$.  Then it is said that $\partial_v^w f $   exists weakly in $K{S^p}[\R_I^n]$,  if there exists a function $ g \in K{S^p}[\R_I^n]$ such that $$< f, \partial_v \varphi> = -< g, \varphi>_{K{S^p}},~~\forall \varphi \in C_c^\infty[\R_I^n].$$ 
We define $\partial_v f =g$. If $ \alpha \in \N_0^\infty$, then $ \partial^\alpha f$  exists weakly in $K{S^p}[\R_I^n]$ if and only if there exists $ g \in K{S^p}[\R_I^n]$  such that $$<f, \partial^\alpha \varphi> = (-1)^\alpha < g, \varphi>_{KS^p},~\forall \varphi \in C_c^\infty[\R_I^n].$$ In this case $ \partial^\alpha f=g$.
      \end{Def}
Since $K{S^p}[\R_I^n]$ is the completion of $L^p[\R_I^n]$, we define the strongly differentiability in $K{S^p}[\R_I^n]$  as was done in $L^p[\R_I^n].$
      \begin{Def}
      Let $ v \in \R^n$ and $f \in K{S^p}[\R_I^n]$. Then $\partial_v^s f $ exists strongly in $K{S^p}[\R_I^n]$, if $\lim\limits_{h \to 0} \partial_{v^{h}} f $ exists in $K{S^p}[\R_I^n]$. We define $\partial_v^s f =\lim\limits_{h \to 0} \partial_{v^{h}} f$.
      \end{Def}
   Strongly differentiable  implies weakly differentiability in $L^p[\R_I^n]$ and weakly convergent  in $L^p[\R_I^n]$ is strongly convergence in $K{S^p}[\R_I^n]$ with compact support. This lead our investigation is more interesting, we want to find a relation between strongly differentiable and weakly differentiable in $K{S^p}[\R_I^n].$
      Easily we can find for $f \in K{S^p}[\R_I^n]$ and $ v \in \R^d,$ if $\partial_{v}^{s} f\in K{S^p}[\R_I^n]$ exists weakly then $\partial_{v} f \in K{S^p}[\R_I^n]$  and they are equal. Now 
      \begin{align*}
      <\partial_{v}^{h} f, \varphi> &= \sum_{r=1}^{\infty} \tau_r \int_{\R^d}\zeta_r(x)\left\{\frac{f(x+hv)-f(x)}{h}\right\} \varphi(x) d\lambda_n(x)\\&= \sum_{k=1}^{\infty} \tau_r \int_{\R^d} \zeta_r(x)\{f(x)\} \frac{\varphi(x-hv)-\varphi(x)}{h}d \lambda_n(x) \\&= \sum_{r=1}^{\infty} \tau_r \int_{\R^d} \zeta_r(x) f(x) \frac{\varphi(x-hv)-\varphi(x)}{h}d \lambda_n(x)\\&=< f, \partial_{-v}^{h} \varphi>_{KS^p}.
      \end{align*}
      If $\partial_{v}^{s}f= \lim\limits_{h \to 0} \partial_{v}^{h} f $ exists in $K{S^p}[\R_I^n]$ and $\varphi \in C_c^\infty[\R^d]$ then 
      \begin{align*}
      < \partial_{v}^{s} f, \varphi> &= \lim\limits_{h \to 0} < \partial_{v}^{h} f , \varphi> \\&= \lim\limits_{h \to 0}< f, \partial_{-v}^{h} \varphi> \\&= - <f, \partial_v \varphi>.
      \end{align*}
      Our main purpose is now to find whether weakly differentiable in $L^p[\R_I^n]$ implies strongly differentiable in $K{S^p}[\R_I^n]$. This comes from the fact that weakly differentiable in $L^p[\R_I^n] $ implies weakly convergence in $L^p[\R_I^n]$, moreover weakly convergence in $L^p[\R_I^n]$ is strongly convergence in $K{S^p}[\R_I^n]$. The fact that $f $ is strongly convergence in $K{S^p}[\R_I^n]$ gives $$ \int_{\R_I^n} \zeta_r(x)[ f_n(x)-f(x)] d\lambda_\infty(x) \to 0 \mbox{~~for~each~} n.$$ This gives us $\lim\limits_{ h \to 0} \partial_{v}^{h} f =0 $ in $K{S^p}[\R_I^n]$.
      \begin{rem}
       Any function in $L^\infty[\R_I^n]$ is weakly derivable in $ L^p[\R_I^n]$ so $$\zeta_r(f)  \in L^p[\R_I^n] \cap L^\infty[\R_I^n]$$ is also in the sense of weak if we consider in weak derivative. 
      \end{rem}
      From this remark we can conclude that weakly differentiable in $L^p[\R_I^n] \cap L^\infty[\R_I^n] $ is also strongly differentiable in $K{S^p}[\R_I^n].$
      \begin{prop}
      If $ \partial_{v} f $ exists in $L_{loc}^{p}[\R_I^n]$ weakly then there exists $ f_n \in C_c^\infty[\R_I^n]$ such that $f_n \to f $ in $K{S^p}[K]$ strongly, i.e $$ \lim\limits_{n \to \infty}||f-f_n||_{K{S^p}[K]}=0$$ and $\partial_v  f_n \to \partial_v f $ in $K{S^p}[K]$ strongly for all $K \subset \subset \R_I^n$.
      \end{prop}
      \begin{proof}
      Proof is easy so we left to the readers.
      \end{proof}
       We can construct above concept for $K{S^p}[\R_I^\infty],$ which we mention now:
        \begin{Def}
      Let $v \in \R^d$ and $f \in K{S^p}[\R_I^\infty]$  then $\partial_v f $ is said to exist weakly in $K{S^p}[\R_I^\infty]$  if there exists a function $ g \in K{S^p}[\R_I^\infty]$ such that $$< f, \partial_v \varphi> = -< g, \varphi>_{K{S^p}},~\forall \varphi \in C_c^\infty[\R_I^\infty].$$ If $ \alpha \in \N_0^\infty$ then $ \partial^\alpha f$ is exists weakly in $K{S^p}[\R_I^\infty]$ if and only if there exists $ g \in K{S^p}[\R_I^\infty]$  such that $$< f, \partial^\alpha \varphi> = (-1)^{\alpha} < g, \varphi>_{KS^p},~\forall \varphi \in C_c^\infty[\R_I^\infty].$$
      \end{Def}
      Since $K{S^p}[\R_I^\infty]$ is completion of $L^p[\R_I^\infty]$, we define strongly differentiable as like of $L^p[\R_I^\infty].$
      \begin{Def}
      Let $ v \in \R^d$ and $f \in K{S^p}[\R_I^\infty]$, then $\partial_v f $ is exists strongly in $K{S^p}[\R_I^\infty]$ if $\lim\limits_{h \to 0} \partial_{v}^{h} f $ exists in $K{S^p}[\R_I^\infty]$.
      \end{Def}
      Same way the strongly differentiable implies weakly differentiable in $L^p[\R_I^\infty]$ and weakly convergent in $L^p[\R_I^\infty]$ is strongly convergence in $K{S^p}[\R_I^\infty]$ with compact support. 
      Easily we can find for $f \in K{S^p}[\R_I^\infty]$ and $ v \in \R^d$, if $\partial_{v}^{s} f \in K{S^p}[\R_I^\infty]$ exists then $\partial_{v}f \in K{S^p}[\R_I^\infty]$  weakly and they are equal. Moreover weakly differentiable in $L^p[\R_I^\infty]$ implies strongly differentiable in $K{S^p}[\R_I^\infty]$. This comes from the fact that weakly differentiable in $L^p[\R_I^\infty] $ implies weakly convergence in $L^p[\R_I^\infty]$, with weakly convergence in $L^p[\R_I^\infty]$ is strongly convergence in $K{S^p}[\R_I^\infty]$.
       Any function in $L^\infty[\R_I^\infty]$ is weakly derivable in $ L^p[\R_I^\infty]$ so $$\zeta_r(f)  \in L^p[\R_I^\infty] \cap L^\infty[\R_I^\infty]$$ is also in the sense of weak if we consider in weak derivative. 
      From this we can conclude that weakly differentiable in $L^p[\R_I^\infty] \cap L^\infty[\R_I^\infty] $ is also strongly differentiable in $K{S^p}[\R_I^\infty].$
      \begin{lem}
      \begin{enumerate}
      \item Suppose $ f \in L_{loc}^{1}[\R_I^\infty]$ and $\partial_{v} f $ exists weakly in $L_{loc}^{1}[\R_I^\infty].$ Then $supp_{m}( \partial_v f) \subset supp_{m}(f),$ where $supp_{m}(f) $ is essential support of $ f $ relative to Lebesgue measure.
      \item If $f $ is continuously differentiable on $ u \subset \R_I^\infty$ then $\partial_{v} f = \partial _{u}(f) $ (weakly) a.e. on $f$
      \end{enumerate}
      \end{lem}
      \begin{prop}
      If $ \partial_{v} f $ exists in $L_{loc}^{p}[\R_I^\infty]$ weakly then there exists $ f_n \in C_c^\infty[\R_I^\infty]$ such that $f_n \to f $ in $K{S^p}[K]$ strongly, i.e $$ \lim\limits_{n \to \infty}||f-f_n||_{K{S^p}[K]}=0$$ and $\partial_v  f_n \to \partial_v f $ in $K{S^p}[K]$ strongly for all $K \subset \subset \R_I^\infty.$
      \end{prop}
    \section{ HK-Sobolev spaces}
    The function $f(x)=|x|$  is weak derivable in $KS^p(\R_I^{n})$ which is not strongly derivable in   $KS^p(\R_I^{n})$. This type of functions motivate us to think in a space like Sobolev for  $KS^p(\R_I^{n})$ and $KS^p(\R_I^{\infty}).$\\
     In the one dimensional case the HK-Sobolev space $W{S^{k, p}}[\R]$ for $ 1 \leq p \leq \infty$ is defined as the subset of functions $ f $ in $K{S^p}[\R]$ such that $ f $ and its weak derivatives upto order $k$ have a finite $K{S^p}$ norm.\\
     In one dimensional problem it is enough to assume that $f^{(k-1)},$ the $(k-1)th$ derivative of the function $f $ is differentiable almost every where. That is $$W{S^{k,p}}[\R] = \{ f(x):~\D^k f(x) \in K{S^p}[\R]\}.$$ For multi-dimensional case the transition to multiple dimensions entails more difficulties, starting with the definition itself. The requirement that $f^{(k-1)}$ be the integral of $f^{(k) }$ does not generalize, and the simplest solution is to consider derivatives in the sense of distribution theory.\\
     A formal definition  we now state as: 
      Let $ k \in \N,~~1 \leq p \leq \infty.$ The HK-Sobolev space $ W{S^{k,p}}[\R_I^n]$ is defined as the set of all functions $f $ on $\R_I^n$ such that for every multi-index $ \alpha $ with $|\alpha| \leq k,$ the mixed partial derivative $$f^{(\alpha)} = \frac{\partial^{|\alpha|} f}{\partial x_{1}^{\alpha_1}....\partial x_{n}^{\alpha_n}}$$ exists in the weak sense in $K{S^p}[\R_I^n]$ that is $||f^{(\alpha)}||_{K{S^p}} < \infty.$\\
      Therefore the HK-Sobolev  space $W{S^{k,p}}[\R_I^n]$ is the space  $$W{S^{k,p}}[\R_I^n] = \{ f \in K{S^p}[\R_I^n]:~~\D^{\alpha} f \in K{S^p}[\R_I^n],~\forall |\alpha| \leq k \}.$$ We called $k $ as the order of the HK-Sobolev  space $W{S^{k,p}}[\R_I^n].$ We define a norm for $W{S^{k,p}}[\R_I^n] $ as: $$||f||_{W{S^{k,p}}[\R_I^n]}  =\left\{\begin{array}{c}\left(\sum\limits_{|\alpha| \leq k}||\D^{\alpha} f ||_{KS^p}^{p} \right)^{\frac{1}{p}}, \mbox{~for~} 1\leq p<\infty;\\
  \max\limits_{|\alpha| \leq k}||\D^k f||_{K{S^\infty}}~~, \mbox{~for~} p=\infty \end{array}\right.$$ For $k=1$ 
  \begin{align*}
  ||f||_{W{S^{1,p}}[\R_I^n]} &= \left(||f||_{K{S^p}[\R_I^n]}^{p} + ||\D f||_{K{S^p}[\R_I^n]}\right)^{\frac{1}{p}}\\&= \left\{ \sum_{r=1}^{\infty} \tau_r \left| \int_{\R_I^n} \zeta_r(x)f(x)d \lambda_\infty(x)\right|^p + \sum_{r=1}^{\infty} \tau_r \left| \int_{\R_I^n} \zeta_r(x) \D f(x) d\lambda_\infty(x)\right|^p\right\}^{\frac{1}{p}}
  \end{align*}
 and $$ ||f||_{W{S^{1,\infty}}[\R_I^n]} = \sup_{ r \geq 1}\left|\int_{\R_I^n} \zeta_r(x) f(x) d \lambda_\infty(x)\right| + \sup_{ r \geq 1 }\left| \int_{\R_I^n} \zeta_r (x) \D f(x) d \lambda_\infty(x)\right|.$$ We can consider equivalent norms \begin{align*}
 	&||f||_{W{S^{1,p}[\R_I^n]}} =\left( ||f||_{K{S^p}[\R_I^n]}^{p} + \sum_{j=1}^{n} ||\D_j f ||_{K{S^p}[\R_I^n]}^{p}\right)^{\frac{1}{p}},\\ &||f||_{W{S^{1,p}[\R_I^n]}}= ||f||_{K{S^p}[\R_I^n]} + \sum_{j=1}^{n} ||\D_j f ||_{K{S^p}[\R_I^n]}\end{align*}
  when $ 1 \leq p < \infty$  and $$||f||_{W{S^{1, \infty}}}= \max \{||f||_{L^\infty[\R_I^n]}, ||\D f||_{L^\infty[\R_I^n]},\dots, ||\D_n f||_{L^\infty[\R_I^n]} \}.$$ 
  \begin{rem}
  The elements of  $W{S^{k,p}}[\R_I^n]$  are equal a.e.
  \end{rem}
  \subsection{{\bf Completeness of HK-Sobolev  Spaces}:}
   A sequence $(f_i) $ of functions $f_i \in W{S^{k,p}}[\R_I^n]~~i=1,2,...$ converges  to a function $ f \in W{S^{k,p}}[\R_I^n] $ if for every $ \epsilon > 0 $ there exists $i_\epsilon $ such that $$||f_i - f||_{W{S^{k,p}}[\R_I^n]}< \epsilon \mbox{~ when~} i \geq i_\epsilon.$$
   Equivalently $$\lim\limits_{i \to \infty}||f_i - f||_{W{S^{k,p}}[\R_I^n]}=0$$ A sequence $(f_i) $ is a Cauchy sequence in $W{S^{k,p}}[\R_I^n] $ if for every $ \epsilon >0 $ there exists $i_\epsilon $ such that $$||f_i - f_j ||_{W{S^{k,p}}[\R_I^n]} < \epsilon  \mbox{~when~} i, j \geq i_\epsilon .$$
   \begin{thm}
   $W{S^{k,p}}[\R_I^n] $ is Banach space.
   \end{thm}
   \begin{proof}
   First we prove $||.||_{W{S^{k,p}}[\R_I^n]} $ is a norm.
   \begin{enumerate}
   \item $||f||_{W{S^{k,p}}[\R_I^n]}=0 ~~\Leftrightarrow f=0 $ a.e. in $\R_I^n.$\\$
  ||f||_{W{S^{k,p}}[\R_I^n]}=0 \Rightarrow ||f||_{K{S^p}[\R_I^n]}=0$  which implies $ f=0$  a.e. in $ \R_I^n.$\\
   Now  $ f=0 $ a.e. in $\R_I^n $, implies $$ \int_{\R_I^n} \D^\alpha f \varphi d \lambda_\infty = (-1)^{| \alpha|} \int_{\R_I^n} f \D^\alpha \varphi d \lambda_\infty =0 \mbox{~for~all~} \varphi \in C_0^\infty[\R_I^n].$$
   As $ f \in L_{loc}^{1}[\R_I^n] $ satisfies $ \int_{\R_I^n} f \varphi d \lambda_\infty = 0 $ for every  $ \varphi \in C_0^\infty[\R_I^n]$ then $ f=0 $ a.e. in $\R_I^n.$ This implies $ \D^\alpha f =0 $ a.e. in $\R_I^n$ for all $\alpha,~~|\alpha| \leq k.$
   \item $||\alpha f ||_{W{S^{k,p}}[\R_I^n]}= |\alpha|||f||_{W{S^{k,p}}[\R_I^n]},~~\alpha \in \R.$
   \item The triangle inequality for $ 1 \leq p < \infty $ follows from elementary inequality $( a+ b)^\alpha \leq a^\alpha + b^\alpha,~0< \alpha \leq 1 $ and Minkowski's inequality.
   \end{enumerate}
   Now, let $ (f_i) $ be Cauchy sequence in $W{S^{k,p}}[\R_I^n] $, since  $$|| \D^\alpha f_i - \D^\alpha f_j||_{K{S^p}[\R_I^n]} \leq ||f_i - f_j||_{W{S^{k,p}}[\R_I^n]}~, |\alpha | \leq k$$ it follows that $ ( \D^\alpha f_i) $ is Cauchy in $K{S^p}[\R_I^n],~|\alpha| \leq k,$ next follow the completeness of $K{S^p}[\R_I^n] $ implies that there exists $ f_{\alpha} \in K{S^p}[\R_I^n] $ such that $\D^\alpha f_i  \to f_{\alpha}$ in $K{S^p}[\R_I^n].$
   \end{proof}
   \begin{rem}
    $W{S^{k,p}}[\R_I^n] ,~1 \leq p < \infty $ is separable, however $W{S^{1, \infty}}[\R_I^n] $ is not separable.
   \end{rem}
    The space $W{S^{k,2}}[\R_I^n]$ is a Hilbert space with the inner product $$< f, g> _{W{S^{k,2}}[\R_I^n]} = \sum_{| \alpha| \leq k }< \D^\alpha f, \D^\alpha g > _{K{S^2}[\R_I^n]},$$ where $$<\D^\alpha f, \D^\alpha g > _{K{S^2}[\R_I^n]}= \sum_{r=1}^{\infty} \tau_r  \int_{\R_I^n} \zeta_r(x) \D^\alpha f \D^\alpha g d \lambda_\infty(x).$$
   Observe that $||f||_{W{S^{k,2}}[\R_I^n]} = <  f, f>_{W{S^{k,2}}[\R_I^n]}^{\frac{1}{2}}.$
   \begin{thm}
   For $1 \leq p < \infty,$ we have 
   \begin{enumerate}
   \item If $f_n \to f $ weakly in $W^{1,p}[\R_I^n] $ then $f_n \to f $ strongly in $W{S^{1,p}}[\R_I^n]$ i.e. every weakly compact subset of $W^{1,p}[\R_I^n] $ is compact in $W{S^{1,p}}[\R_I^n].$
   \item If $ 1 < p < \infty $ then $W{S^{k,p}}[\R_I^n]$ is uniformly convex.
   \item If $ 1 < p < \infty $ then $W{S^{k,p}}[\R_I^n]$ is reflexive.
   \item  $W{S^{k,\infty}}[\R_I^n] \subset W{S^{k,p}}[\R_I^n]$ for $ 1 \leq p < \infty.$
   \end{enumerate}
   \end{thm}
   \begin{proof}
   (1) Using  \cite[Theorem 1.38]{JH} and  $(f_n) $ is weakly convergence in $W^{1,p}[\R_I^n] $ to  $ f \in W^{1,p}[\R_I^n], $ we have that $ f_n \to f \in L^p[\R_I^n]$ weakly. This implies $f_n \to f \in K{S^p}[\R_I^n] $ strongly. Consequently,
$
 \int_{\R_I^n}\zeta_r(x)[ f_n(x) - f(x)] d \lambda_\infty(x) \to 0
$ and so
  $
 \int_{\R_I^n}\zeta_r(x)[\D^\alpha f_n(x) - \D^\alpha f(x)] d \lambda_\infty(x) \to 0.
 $
   Therefore, $ f_n \to f \in W{S^{1,p}}[\R_I^n]. $\\
   (2) Let $ T: W{S^{k,p}}[\R_I^n] \to K{S^p}[\R_I^n]$, defined as $ x \to (\D^\alpha x)_{|\alpha| \leq k}$,  be a closed and isometric embedding. Since  $K{S^p}[\R_I^n] $ is uniformly convex for $ 1 < p < \infty,$ so  is any closed subspace, and hence as $W{S^{k,p}}[\R_I^n] $ is isometric to its image under $T,$ it follows that  $W{S^{k,p}}[\R_I^n] $ is uniformly convex for  these $p.$\\
   (3) Follows from part (2).\\
   (4) Let $ f \in W{S^{k, \infty}}[\R_I^n] .$    This implies that $| \int_{\R_I^n} \zeta_r(x) \D^\alpha f(x) d \lambda_\infty(x)| $ is uniformly bounded for all $n.$ Then $|\int_{\R_I^n} \zeta_r(x) \D^\alpha f(x) d \lambda_\infty (x) |^p $ is uniformly bounded for each $p,~1 \leq p < \infty.$ So, it is clear $$\left[\sum_{r=1}^{\infty}\tau_r \left| \int_{\R_I^n} \zeta_r(x)\D^\alpha f(x) d \lambda_\infty(x)\right|^p\right]^{\frac{1}{p}} < \infty. $$ Therefore, $ f \in W{S^{k,p}}[\R_I^n].$
    \end{proof}
   \begin{thm}
   $W^{k,p}[\R_I^n] \subset W{S^{k,p}}[\R_I^n] $ as continuous dense embedding for $ 1 \leq p < \infty.$
   \end{thm}
   \begin{proof}
    Let $ f \in W^{k,p}[\R_I^n] $ for $ 1 \leq p < \infty.$ Then we have 
    \begin{align*}
    ||f||_{W{S^{k,p}}[\R_I^n]} &= \left( \sum_{|\alpha| \leq k}|| \D^\alpha f ||_{K{S^p}}^{p}\right)^{\frac{1}{p}} \\&= \left( \sum_{|\alpha| \leq k}  \sum_{r=1}^{\infty} \tau_r \left| \int_{\R_I^n} \zeta_r(x) \D^\alpha f(x) d \lambda_\infty(x)\right|^p\right)^{\frac{1}{p}} \\&\leq \left[ \sum_{|\alpha| \leq k }\sum_{r=1}^{\infty} \tau_r  \int_{\R_I^n} \zeta_r(x)| \D^\alpha f(x)|^p d \lambda_\infty (x))^{\frac{1}{p}}\right] \\& \leq \sup_{|\alpha| \leq k}\left( \int_{\R_I^n} \zeta_r(x) | \D^\alpha f(x) |^p d \lambda_\infty(x)\right)^{\frac{1}{p}}\leq ||f||_{W^{k,p}[\R_I^n]}.
    \end{align*}
   \end{proof}
   \begin{thm}
   $W{S^{1,p}}[\R_I^n] \to K{S^p}[\R_I^n]$ as continuous embedding for $ 1 \leq p < \infty.$
   \end{thm}
   \begin{proof}
   As \cite{MV}, we have $W^{1,p}[\R_I^n] \to L^q[\R_I^n] $ for $ 1\le p \leq q < \infty$. Also, $L^q[\R_I^n] \subset K{S^p}[\R_I^n]$ as continuous dense for $1 \leq p < \infty$. So, $W^{1,p}[\R_I^n] \to K{S^p}[\R_I^n] $ for $1 \leq p < \infty.$ 
   We need to prove $W{S^{1,p}}[\R_I^n] \to K{S^p}[\R_I^n]$. For this we find $$||f||_{K{S^p}[\R_I^n]} \leq ||f||_{W{S^{1,p}}[\R_I^n]} $$ for $ f \in W{S^{1,p}}[\R_I^n]$, which gives our result.
   \end{proof}
   \subsection{HK-Sobolev  Spaces on $\R^\infty$:}
   In this section we will discuss about $W{S^{k,p}}[\R_I^\infty]$. 
   As $W{S^{k,p}}[\R_I^n] \subset W{S^{k,p}}[\R_I^{n+1}]$,  we can define $$W{S^{k,p}}[\widehat{\R_I^\infty}] = \bigcup_{n=1}^{\infty} W{S^{k,p}}[\R_I^n].$$
   \begin{Def}
    We say that for $1 \leq p < \infty,$ a measurable function $ f \in W{S^{k,p}}[\R_I^\infty] $, if there exists a Cauchy sequence $\{f_n\} \subset W{S^{k,p}}[\widehat{\R_I^\infty}]$ with $f_n \in W{S^{k,p}}[\R_I^n] $ and $$\lim\limits_{n \to \infty} \D^\alpha f_n(x) = \D^\alpha f(x),~~\lambda_\infty- a.e.$$
   \end{Def}
   \begin{Def}
    Let $f \in W{S^{k,p}}[\R_I^\infty],$ we define the integral by $$\int_{\R_I^\infty} \D^\alpha f(x) d \lambda_\infty(x) = \lim\limits_{n \to \infty} \int_{\R_I^\infty}\D^\alpha f_n(x) d \lambda_\infty(x) $$ where $f_n \in W{S^{k,p}}[\R_I^\infty] $ for all $n$ and the family $\{f_n\}$ is a Cauchy sequence.
   \end{Def}
   \begin{thm}
    $W{S^{k,p}}[\widehat{\R_I^\infty}]= W{S^{k,p}}[\R_I^\infty].$
   \end{thm}   
    We define  HK-Sobolev  space $W{S^{k,p}}[\R_I^\infty] $  as  $$W{S^{k,p}}[\R_I^\infty] = \{ f\in K{S^p}[\R_I^\infty]:~\D^{\alpha} f \in K{S^p}[\R_I^\infty],~\forall |\alpha| \leq k \}.$$ We called $k $ as the order of the HK-Sobolev  space $W{S^{k,p}}[\R_I^\infty]$. We define a	 norm for $W{S^{k,p}}[\R_I^\infty] $ as: $$||f||_{W{S^{k,p}}[\R_I^\infty]}  =\left\{\begin{array}{c}\left(\sum\limits_{|\alpha| \leq k}||\D^{\alpha} f ||_{KS^p}^{p} \right)^{\frac{1}{p}}, \mbox{~for~} 1\leq p<\infty;\\
  \max\limits_{|\alpha| \leq k}||\D^k f||_{K{S^\infty}}~~, \mbox{~for~} p=\infty \end{array}\right.$$ For $k=1$
  \begin{align*}
  ||f||_{W{S^{1,p}}[\R_I^\infty]} &= \left(||f||_{K{S^p}[\R_I^\infty]}^{p} + ||\D f||_{K{S^p}[\R_I^\infty]}\right)^{\frac{1}{p}}\\&= \left\{ \sum_{n=1}^{\infty} \tau_r \left| \int_{\R_I^\infty} \zeta_r(x)f(x)d \lambda_\infty(x)\right|^p + \sum_{r=1}^{\infty} \tau_r \left| \int_{\R_I^\infty} \zeta_r(x) \D f(x) d\lambda_\infty\right|^p\right\}^{\frac{1}{p}}
  \end{align*}
  and $$ ||f||_{W{S^{1,\infty}}[\R_I^\infty]} = \sup_{ r \geq 1}\left|\int_{\R_I^\infty} \zeta_r(x) f(x) d \lambda_\infty(x)\right| + \sup_{ r \geq 1 }\left| \int_{\R_I^\infty} \zeta_r (x) \D f(x) d \lambda_\infty(x)\right|.$$ We can consider equivalent norms 
  \begin{align*}
  	&||f||_{W{S^{1,p}[\R_I^\infty]}} =\left( ||f||_{K{S^p}[\R_I^\infty]}^{p} + \sum_{j=1}^{n} ||\D_j f ||_{K{S^p}[\R_I^\infty]}^{p}\right)^{\frac{1}{p}},\\
  	&||f||_{W{S^{1,p}[\R_I^\infty]}}= ||f||_{K{S^p}[\R_I^\infty]} + \sum_{j=1}^{n} ||\D_j f ||_{K{S^p}[\R_I^\infty]}\end{align*} when $ 1 \leq p < \infty$  and $$||f||_{W{S^{1, \infty}}}= \max \{||f||_{L^\infty[\R_I^\infty]}, ||\D f||_{L^\infty[\R_I^\infty]},.., ||\D_n f||_{L^\infty[\R_I^\infty]} \}.$$ 
  \begin{rem} The functions 
   $W{S^{k,p}}[\R_I^\infty]$  are equal a.e.
  \end{rem}
  \begin{thm}
  Let $f \in W{S^{k,p}}[\R_I^\infty],$  then  $$\int_{\R_I^\infty} \D^\alpha f(x) d \lambda_\infty(x) = \lim\limits_{n \to \infty}\int_{\R_I^\infty} \D^\alpha f_n(x) d \lambda_\infty(x) $$ where $f_n \in W{S^{k,p}}[\R_I^\infty] $ for all $n$ and the family $\{f_n\}$ is a Cauchy sequence.
  \end{thm}
   For the proof of this theorem, we follow the Section 3.1.
   \begin{thm}\label{th311}
   For $ 1 \leq p < \infty,$ we have 
   \begin{enumerate}
   \item If $f_n \to f $ weakly in $W^{1,p}[\R_I^\infty] $ then $f_n \to f $ strongly in $W{S^{1,p}}[\R_I^\infty]$ i.e. every weakly compact subset of $W^{1,p}[\R_I^\infty] $ is compact in $W{S^{1,p}}[\R_I^\infty]$.
   \item If $ 1 < p < \infty $ then $W{S^{k,p}}[\R_I^\infty]$ is uniformly convex.
   \item If $ 1 < p < \infty $ then $W{S^{k,p}}[\R_I^\infty]$ is reflexive.
   \item  $W{S^{k,\infty}}[\R_I^\infty] \subset W{S^{k,p}}[\R_I^\infty]$ for $ 1 \leq p < \infty .$
   \end{enumerate}
   \end{thm}
   \begin{proof}
    (1) Let $\{f_n\} $ be a weakly convergence sequence in $W^{1,p}[\R_I^\infty] $ with limit $f.$ Then $$\int_{\R_I^n} \zeta_r(x)| \D^\alpha f_n(x) - \D^\alpha f(x)| d \lambda_\infty(x) \to 0 $$ for each $r$. Now, since each $f_n \in W{S^{1,p}}[\R_I^n] $, it follows $$\lim\limits_{n \to \infty}\int_{\R_I^n} \zeta_r(x)|\D^\alpha f_n(x)- \D^\alpha f(x)| d \lambda_\infty(x) \to 0.$$\\
     (2) As $W^{k,p}[\R_I^n] $ is uniformly convex for each $n $ and that is dense and compactly embedded in $W{S^{k,p}}[\R_I^\infty]$ for all $p,~1 \leq p \leq \infty.$ So, $\bigcup\limits_{n=1}^{\infty} W^{k,p}[\R_I^n] $ is uniformly convex for each $n $ and that is dense and compactly embedded in $\bigcup\limits_{n=1}^{\infty} W{S^{k,p}}[\R_I^\infty]$ for all $p,~1 \leq p \leq \infty.$\\
     However $W^{k,p}[\widehat{\R_I^\infty}]= \bigcup\limits_{n=1}^{\infty} W^{k,p}[\R_I^n]$. That is $W^{k,p}[\widehat{\R_I^\infty}] $ is uniformly convex, dense and compactly embedded in $W{S^{k,p}}[\widehat{\R_I^\infty}] $ for all $p,~1 \leq p \leq \infty$. As $W{S^{k,p}}[\R_I^\infty] $ is closure of $W{S^{k,p}}[\widehat{\R_I^\infty}]$. Therefore $W{S^{k,p}}[\R_I^\infty] $ is uniformly convex.\\
      (3) From (2) we have  $W{S^{k,p}}[\R_I^\infty]$ is reflexive for $ 1 < p < \infty.$\\
      (4) Let $ f \in W{S^{k,p}}[\R_I^\infty].$ This implies $$\left|\int_{\R_I^\infty} \zeta_r(x) \D^\alpha f(x) d \lambda_\infty (x)\right|$$ is uniformly bounded for all $r.$ It follows that $\left |\int_{\R_I^\infty} \zeta_r(x) \D^\alpha f(x) d \lambda_\infty (x)\right|^p$ is uniformly bounded for $ 1 \leq p < \infty.$ It is clear from the definition of $W{S^{k,p}}[\R_I^\infty]$ that $$\left[ \sum_{r=1}^{\infty}\tau_r\left| \int_{\R_I^\infty} \zeta_r(x) \D^\alpha f(x) d \lambda_\infty(x)\right|^p\right]^{\frac{1}{p}} \leq M||f||_{W{S^{k,p}}[\R_I^\infty]} < \infty.$$ So, $f \in W{S^{k,p}}[\R_I^\infty].$
   \end{proof}
   \begin{thm}
   $W{S^{1,p}}[\R_I^\infty] \to K{S^p}[\R_I^\infty]$ as continuous embedding for $ 1 \leq p < \infty.$
   \end{thm}
   \begin{proof}
   As $W{S^{1,p}}[\R_I^n] \longrightarrow K{S^p}[\R_I^n]$ as continuous embedding for $ 1 \leq p < \infty.$ So, $\bigcup\limits_{n=1}^{\infty}W{S^{1,p}}[\R_I^n] \longrightarrow  \bigcup\limits_{n=1}^{\infty}K{S^p}[\R_I^n]$ as continuous embedding for $ 1 \leq p < \infty$. Therefore $W{S^{1,p}}[\widehat{\R_I^\infty}] \longrightarrow K{S^p}[\widehat{\R_I^\infty}] $ for $ 1 \leq p < \infty.$ 
   Hence, $W{S^{1,p}}[\R_I^\infty] \longrightarrow K{S^p}[\R_I^\infty]$ as continuous embedding for $ 1 \leq p < \infty.$
   \end{proof}
   \section{HK-Sobolev  spaces through Bessel Potential}
   \begin{Def} \cite{WM}
   \begin{enumerate}
   \item $\mcS[\R_I^n]= \{f \in C_c^\infty[\R_I^n]~~: \sup|x^\beta \partial^\alpha f | < \infty \} $ for all multi-indices $\alpha, \beta.$
   \item $\mcS^{*}[\R_I^n] $ is the set of sequentially continuous functionals on the space $\mcS[\R_I^n].$
    \end{enumerate}
   \end{Def}
   For $ s \in \R,$ Bessel potential of order $s $ to be the sequentially continuous bijective linear operator $\mcP^s: \mcS[\R_I^n] \to \mcS[\R_I^n]$ by $$\mcP^s f(x)= \int_{\R_I^n}( 1 +|\xi|^2)^{\frac{s}{2}}\widehat{f}(\xi) e^{i 2\pi \xi.x}d \xi$$ for $ x \in \R_I^n.$ That is $$\mcF_{x \to \xi}\{\mcP^s f(x)\}=( 1 +|\xi|^2)^{\frac{s}{2}}\widehat{f}(\xi).$$
  \begin{Def}\cite{WM}
  For $ s \in \R,$  the Sobolev space through Bessel potential is $${W_{\mcB}^{s,p}}[\R_I^n]= \{ f \in \mcS^{*}[\R_I^n]:~~\mcP^s f \in L^2[\R_I^n] \},$$ where
  $\mcP^s $ is a kind of differential operator of order $s.$
  \end{Def}
    Recalling that $\mcP^{s+t}= \mcP^s.\mcP^t,~(\mcP^s)^{-1}= \mcP^{-s},~\mcP^0= $ Identity operator. It is found that $<\mcP^s f, g>=<f, \mcP^s g>_{K{S^2}} $ and $<\mcP^s f, g> = < f, \mcP^s g>_{K{S^2}}$ for all $f, g \in \mcS[\R_I^n],$ giving a natural extension of Bessel potential to a linear operator $\mcP^s: \mcS^{*} \to \mcS^{*}$ by $<\mcP^s f, \varphi>= < f, \mcP^s \varphi> $ for all $ \varphi \in \mcS[\R_I^n].$\\
     We define for $ s \in \R,$ $$W{S_{\mcB}^{s,p}}[\R_I^n]= \{ f \in \mcS^{*}[\R_I^n]:~\mcP^s f \in K{S^2}[\R_I^n] \}.$$
   \begin{thm}
   Let $h \in \mcS^{*}[\R_I^n].$ Then $h \in A$ there exists $ g \in A $ such that $$< h, \varphi>= < g, \varphi>_{K{S^2}[\R_I^n]} \mbox{~for~ all~}\varphi \in \mcS[\R_I^n].$$
   \end{thm} 
   \begin{proof}
    From \cite{WM}, for  $h \in \mcS^{*}[\R_I^n],$ then $h \in A$ there exists $ g \in A $ such that $$< h, \varphi>= < g, \varphi>_{L^2[\R_I^n]}  \mbox{~for~all~} \varphi \in \mcS[\R_I^n].$$
    Let $ f \in L^2[\R_I^n]$. Then
    \begin{align*}
    ||f||_{K{S^2}[\R_I^n]} &= \left[ \sum_{r=1}^{\infty}\tau_r\left|\int_{\R_I^n} \zeta_r(x)f(x) d \lambda_\infty(x)\right|^2\right]^{\frac{1}{2}}\\& \leq \left[ \sum_{r=1}^{\infty} \tau_r \int_{\R_I^n} \zeta_r(x)|f(x)|^2 d \lambda_\infty(x)\right]^{\frac{1}{2}}\\&\leq \sup\left[\int_{\R_I^n} \zeta_r(x)|f(x)|^2 d \lambda_\infty(x)\right]^{\frac{1}{2}}\\&\leq ||f||_{L^2}.
    \end{align*}
    Therefore $ f \in K{S^2}[\R_I^n]$. Hence we can conclude $$< h, \varphi>= < g, \varphi>_{K{S^2}[\R_I^n]} \mbox{~for~ all~}\varphi \in \mcS[\R_I^n].$$
   \end{proof}
     $W{S_{\mcB}^{s,p}}[\R_I^n]$ equipped with the inner product $$<f,g>_{W{S_{\mcB}^{s,p}}[\R_I^n]}=< \mcP^s f, \mcP^s g>_{K{S^2}[\R_I^n]}$$ and the norm induced by $$< f,f>_{W{S_{\mcB}^{s,p}}[\R_I^n]}= ||\mcP^s f||_{K{S^2}[\R_I^n]}.$$ The Bessel Potential $\mcP^s:W{S_{\mcB}^{s,p}}[\R_I^n] \to K{S^2}[\R_I^n] $ is a unitary isomorphism and in particular $\mcP^0 f= f$ so, $W{S_{\mcB}^{0,p}}[\R_I^n] = K{S^2}[\R_I^n].$
    \begin{rem}
    $W{S_{\mcB}^{s,p}}[\R_I^n]$ is separable Hilbert space.
    \end{rem}
    \begin{thm}
    $W{S_{\mcB}^{s,p}}[\R_I^n] $ contains $W_{\mcB}^{s,p}[\R_I^n]$ as continuous dense embedding.
    \end{thm}
    \begin{proof}
    Let $ f \in W{S_{\mcB}^{s,p}}[\R_I^n].$ Then
    \begin{align*}
    ||f||_{W{S_{\mcB}^{s,p}}[\R_I^n]} &=||\mcP^s f||_{K{S^2}[\R_I^n]}\\&=\left[ \sum_{r=1}^{\infty}\tau_r \left|\int_{\R_I^n}\zeta_r(x)(\mcP^s f)(x)d \lambda_\infty(x)\right|^2\right]^{\frac{1}{2}}\\&\leq \left[\int_{\R_I^n}|(\mcP^s f)(x)|^2 d \lambda_\infty(x)\right]^{\frac{1}{2}}\\&\leq ||f||_{W_{\mcB}^{s,p}[\R_I^n]}.
    \end{align*}
    Hence the result.
    \end{proof}
 \begin{thm}\label{th46}
	\begin{enumerate}
		\item If $f_n \to f $ weakly in $W^{^{s,p}}_{\mcB}[\R_I^n]$ then $f_n \to f $ strongly in  $W{S_{\mcB}^{s,p}}[\R_I^n] $.
		\item If $ 1 < p < \infty$ then $W{S_{\mcB}^{s,p}}[\R_I^n]$ is uniformly convex.
		\item If $ 1 < p < \infty$ then $W{S_{\mcB}^{s,p}}[\R_I^n]$ is reflexive.
	\end{enumerate}
	
\end{thm}
\begin{proof}
	Proof are similar technique as Theorem \ref{th311}.
\end{proof}
\begin{thm}
\begin{enumerate}
\item $\D[\R_I^n] $ is dense in $W{S_{\mcB}^{s,p}}[\R_I^n].$
\item $\mcS[\R_I^n]$ is dense in $W{S_{\mcB}^{s,p}}[\R_I^n]$
\item If $s \leq t$ then $W{S_{\mcB}^{t,p}}[\R_I^n] \subset W{S_{\mcB}^{s,p}}[\R_I^n]$ and $||f||_{W{S_{\mcB}^{s,p}}} \leq ||f||_{W{S_{\mcB}^{t,p}}}.$
\end{enumerate}
\end{thm}
\begin{proof}
 (1) $\D[\R_I^n] $ is dense in $W_{\mcB}^{s,p}[\R_I^n]$ and $W_{\mcB}^{s,p}[\R_I^n]$ is dense in  $W{S_{\mcB}^{s,p}}[\R_I^n].$ Hence $\D[\R_I^n] $ is dense in $W{S_{\mcB}^{s,p}}[\R_I^n].$\\
For (2) and (3) we can follow the similar arguments like part (1).
\end{proof}
\begin{Def} For any closed set $ F \subset \R_I^n,$ the associated HK-Sobolev space of order $s$, denoted by $B^{s,F}$ is defined by $$B^{s,F}=\{f \in W{S_{\mcB}^{s,p}}[\R_I^n]~: supp f \subset F\}.$$
\end{Def}
\begin{thm}\label{th49}
	$B^{s,F}$ is a closed subspace of $W{S_{\mcB}^{s,p}}[\R_I^n]$.
\end{thm}
\begin{proof}
	Let $\{f_i\}_{i=1}^{\infty} $ be in $B^{s,F}$ converges to $f $ in $W{S_{\mcB}^{s,p}}[\R_I^n]$.\\
	If $\mcF \in \mcD(F^c), $ and let $\widehat{\mcF}$ denote the extension of $\mcF $ to $\mcD[\R_I^n] $ by zero. Then 
	\begin{align*}
	<u_{|F_c}, \mcF> &= < f, \widehat{\mcF}> \\ &= < f-f_i, \widehat{\mcF}> + < f_i, \widehat{\mcF}> \\&=< f- f_i, \widehat{\mcF}> \\&=< \mcP^s f -\mcP^s f_i, \widehat{\mcF}>_{K{S^p}[\R_I^n]}.
	\end{align*}
	Using Cauchy-Schwartz inequality we get $< f_{|F^c}, \widehat{\mcF}>=0 $ for all $ f \in \mcD(F^c)$. This implies $supp f \subset F.$
\end{proof}
\subsection{HK-Sobolev  spaces through Bessel Potential on $\R^\infty$}
 As $S[\R_I^n] \subset S[\R_I^{n+1}]$ so we can define $S[\widehat{\R_I^\infty}]= \bigcup\limits_{n=1}^{\infty}S[\R_I^n]$. 
 \begin{Def}
   \begin{enumerate}
   \item $\mcS[\R_I^\infty]= \{f \in C_c^\infty[\R_I^\infty]~~: \sup|x^\beta \partial^\alpha f | < \infty \} $ for all multi-indices $\alpha, \beta.$
   \item $\mcS^{*}[\R_I^\infty] $ is the set of sequentially continuous functionals on the space $\mcS[\R_I^\infty].$
    \end{enumerate}
   \end{Def}
    \begin{Def}
  For $ s \in \R,$ Bessel potential of order $s $ to be the sequentially continuous bijective linear operator $\mcP^s: \mcS[\R_I^\infty] \to \mcS[\R_I^\infty]$ by $$\mcP^s f(x)= \int_{\R_I^\infty}( 1 +|\xi|^2)^{\frac{s}{2}}\widehat{f}(\xi) e^{i 2\pi \xi.x}d \xi$$ for $ x \in \R_I^\infty$. That is $$\mcF_{x \to \xi}\{\mcP^s f(x)\}=( 1 +|\xi|^2)^{\frac{s}{2}}\widehat{f}(\xi).$$
  \end{Def}
  \begin{Def}
  For $ s \in \R,$ the Sobolev space through Bessel potential is $${W_{\mcB}^{s,p}}[\R_I^\infty]= \{ f \in \mcS^{*}[\R_I^\infty]:~\mcP^s f \in L^2[\R_I^\infty] \}.$$
  \end{Def}
   As $\mcP^s $ is one kind of differential operator of order $s.$\\
  We can find easily  $\mcP^{s+t}= \mcP^s.\mcP^t,~(\mcP^s)^{-1}= \mcP^{-s},~\mcP^0=$ Identity operator. It is found that $<\mcP^s f, g>=<f, \mcP^s g>_{K{S^2}} $ and $<\mcP^s f, g> = < f, \mcP^s g>_{K{S^2}}$ for all $f, g \in \mcS[\R_I^\infty],$ giving a natural extension of Bessel potential to a linear operator $\mcP^s: \mcS^{*} \to \mcS^{*}$ by $<\mcP^s f, \varphi>= < f, \mcP^s \varphi> $ for all $ \varphi \in \mcS[\R_I^\infty]$.\\
  As $ W{S_{\mcB}^{s,p}}[\R_I^n] \subset W{S_{\mcB}^{s,p}}[\R_I^{n+1}]$ Thus we can find $W{S_{\mcB}^{s,p}}[\widehat{\R_I^\infty}]= \bigcup\limits_{n=1}^{\infty} W{S_{\mcB}^{s,p}}[\R_I^n]$. We say that for $1 \leq p < \infty,$ a measurable function $ f \in W{S_{\mcB}^{s,p}}[\R_I^\infty] $ if there exists a Cauchy-sequence $\{f_n\} \subset W{S_{\mcB}^{s,p}}[\widehat{\R_I^\infty}] $ with $f_n \in W{S_{\mcB}^{s,p}}[\R_I^n]$ and $$\lim\limits_{n \to \infty} \D^\alpha f_n(x) = \D^\alpha f(x),~\lambda_\infty-a.e.$$ Using same approach of construction of $W{S^{k,p}}[\R_I^\infty]$  we build up $W{S_{\mcB}^{s,p}}[\R_I^\infty].$ \\
   We define for $ s \in \R,$ $$W{S_{\mcB}^{s,p}}[\R_I^\infty]= \{ f \in \mcS^{*}[\R_I^\infty]:~~\mcP^s f \in K{S^2}[\R_I^\infty] \}.$$
   \begin{thm}
   Let $h \in \mcS^{*}[\R_I^\infty],$ then $h \in A$ there exists $ g \in A $ such that $$< h, \varphi>= < g, \varphi>_{K{S^2}[\R_I^\infty]}  \mbox{~for~all~}\varphi \in \mcS[\R_I^\infty].$$
   \end{thm} 
   \begin{proof}
    For  $h \in \mcS^{*}[\R_I^\infty],$ then $h \in A$ there exists $ g \in A $ such that $$< h, \varphi>= < g, \varphi>_{L^2[\R_I^\infty]} \mbox{~for~all~}\varphi \in \mcS[\R_I^\infty].$$
    Let $ f \in L^2[\R_I^\infty].$ Then
    \begin{align*}
    ||f||_{K{S^2}[\R_I^\infty]} &= \left[ \sum_{r=1}^{\infty}\tau_r\left|\int_{\R_I^\infty} \zeta_r(x)f(x) d \lambda_\infty(x)\right|^2\right]^{\frac{1}{2}}\\& \leq \left[ \sum_{r=1}^{\infty} \tau_r \int_{\R_I^\infty} \zeta_r(x)|f(x)|^2 d \lambda_\infty(x)\right]^{\frac{1}{2}}\\&\leq \sup\left[\int_{\R_I^\infty} \zeta_r(x)|f(x)|^2 d \lambda_\infty(x)\right]^{\frac{1}{2}}\\&\leq ||f||_{L^2}.
    \end{align*}
    Therefore $ f \in K{S^2}[\R_I^\infty]$. Hence we can conclude $$< h, \varphi>= < g, \varphi>_{K{S^2}[\R_I^\infty]}  \mbox{~~for~all~}\varphi \in \mcS[\R_I^\infty].$$
   \end{proof}
  The space   $W{S_{\mcB}^{s,p}}[\R_I^\infty]$ equipped with the inner product $$<f,g>_{W{S_{\mcB}^{s,p}}[\R_I^\infty]}=< \mcP^s f, \mcP^s g>_{K{S^2}[\R_I^\infty]}$$ and the norm induced by $$< f,f>_{W{S_{\mcB}^{s,p}}[\R_I^\infty]}= ||\mcP^s f||_{K{S^2}[\R_I^\infty]}.$$ Bessel Potential $\mcP^s:W{S_{\mcB}^{s,p}}[\R_I^\infty] \to K{S^2}[\R_I^\infty] $ is a unitary isomorphism and in particular $\mcP^0 f= f,$ so $W{S_{\mcB}^{0,p}}[\R_I^\infty] = K{S^2}[\R_I^\infty].$
    \begin{rem}
    $W{S_{\mcB}^{s,p}}[\R_I^\infty]$ is separable Hilbert space.
    \end{rem}
    \begin{thm}\label{th412}
    $W{S_{\mcB}^{s,p}}[\R_I^\infty] $ contains $W_{\mcB}^{s,p}[\R_I^\infty]$ as continuous dense embedding.
    \end{thm}
    \begin{proof}
     Since $W{S_{\mcB}^{s,p}}[\R_I^n] $ contains $W_{\mcB}^{s,p}[\R_I^n] $ as continuous dense embedding. However $W{S_{\mcB}^{s,p}}[\R_I^\infty] $ is the closure of $\bigcup\limits_{n=1}^{\infty}W{S_{\mcB}^{s,p}}[\R_I^\infty].$ It follows $W{S_{\mcB}^{s,p}}[\R_I^\infty]$ contains $\bigcup\limits_{n=1}^{\infty} W_{\mcB}^{s,p}[\R_I^n]$ which is dense in $W_{\mcB}^{s,p}[\R_I^\infty]$ as it is the closure.
    Hence the result.
    \end{proof}
\begin{thm}
\begin{enumerate}
\item $\D[\R_I^\infty] $ is dense in $W{S_{\mcB}^{s,p}}[\R_I^\infty].$
\item $\mcS[\R_I^\infty]$ is dense in $W{S_{\mcB}^{s,p}}[\R_I^\infty]$
\item If $s \leq t$ then $W{S_{\mcB}^{t,p}}[\R_I^\infty] \subset W{S_{\mcB}^{s,p}}[\R_I^\infty]$ and $||f||_{W{S_{\mcB}^{s,p}}} \leq ||f||_{W{S_{\mcB}^{t,p}}}.$
\end{enumerate}
\end{thm}
\begin{proof}
 Using the similar approach of the proof of  Theorem \ref{th412}, then we get the results.
\end{proof}
\begin{rem}
 In the space $W{S_{\mcB}^{s,p}}[\R_I^\infty] $ also we get similar type of results like Theorem \ref{th46} and Theorem \ref{th49}.
\end{rem}
\section{Application of $W{S^{k,p}}[\R_I^n]$ }
In this section we will find sufficient condition for the solvability of the divergence equation $\nabla.F= f,$ for $f$ is an element of the subspace $K{S^p}[\R_I^n]$ and $m \in \N$, in the HK-Sobolev  space $W{S^{k,p}}[\R_I^n] $,  with the help of Fourier transformation.\\
Recalling $L^1[\R_I^n] \subset K{S^p}[\R_I^n]$ and the second dual of $\{L^1[\R_I^n]\}^{**} = \mathfrak{M}[\R_I^n] \subset K{S^p}[\R_I^n],$ where $\mathfrak{M}[\R_I^n]$ is the space of bounded finitely additive set functions defined on the Borel sets $\mcB[\R_I^n].$\\ Let us define the Fourier transformation on $K{S^p}[\R_I^n] $ by
 $$\mathfrak{f}(f) = \widehat{f}(y)= \int_{\R_I^n} exp\{-2 \pi i < x, y>\} f(x)d \lambda_\infty(x),$$ where $ x \in K{S^p}, y \in K{S^q}$ and $< x, y> $ is the pairing between $K{S^p} $ and $K{S^q}.$\\
It is well known that Schwartz space $\mcS[\R_I^n] $ of test functions is included in $K{S^p}[\R_I^n].$ The restriction of $\mathfrak{f}$ to $\mcS[\R_I^n] $ has an extension by duality to the space $\mcS^{'}[\R_I^n]$ of tempered distribution on $\R_I^n, $ which is a linear operator called Fourier transform and denoted by $\mathit{f} $ of $\mathfrak{f} $ to $\widehat{\mathfrak{f}}.$
\begin{prop}
	\begin{enumerate}
		\item $\mcS^{.}[\R_I^n]$ is contained in $K{S^p}[\R_I^n],~1 \leq p \leq \infty $ and all bounded random measure on $\R_I^n.$
		\item For $ 1 \leq p \leq 2,`\mathfrak{f}$ applies $K{S^p}[\R_I^n] $ into $K{S^q}[\R_I^n] $ and there exists a positive real number $M_p$ such that $||\widehat{\mathfrak{f}}||_p \leq M_p||\mathfrak{f}||_p,~\mathfrak{f} \in K{S^p}[\R_I^n].$
		\item $\mathfrak{f}$ applies $K{S^2}[\R_I^n] $ onto itself and $||\widehat{\mathfrak{f}}||_2= ||\mathfrak{f}||_2,~~\mathfrak{f} \in K{S^2}[\R_I^n].$
	\end{enumerate}
\end{prop}
Let $\mu_\infty$ be a fixed bounded Radon measure on $\R_I^n.$ For any real number $t>0 $ and any complex function $f $ on $\R_I^n$ which is  continuous  and with compact support, we set 
\begin{align*}
\mu_{\infty_t(f)}&= \int_{\R_I^n} f(x)d \mu_{\infty_t(x)}\\&= \int_{\R_I^n} f(tx)d \mu_\infty(x).
\end{align*}
\begin{prop}
	Suppose $1 \leq p \leq \infty.$ For any real number $t>0 $ and any element $f $ of $K{S^p}[\R_I^n]$ the function $M_{\mu_\infty}^{t}(f)$ is defined by 
	\begin{align*}
	M_{\mu_\infty}^{t} f(x)&= \mu_{\infty_t} * f(x)\\&= \int_{\R_I^n} f(x-ty) d \mu_\infty(y)
	\end{align*}
	represents an element of $K{S^p}[\R_I^n]$ such that $||M_{\mu_\infty}^{t} f|| \leq C|\mu_\infty|[\R_I^n] ||f||,$ where $C$ is a real number not depending on $(f,t).$ \\
	If $\int_{\R_I^n} d \mu_\infty(x)=1 $ and $ f \in K{S^p}[\R_I^n]$ with $ 1 \leq p \leq \infty$ then for any real number $t>0,~M_{\mu_\infty}^{t} f(x) $ as a mean of $f $ on the subset $x-t\mbox{supp}(\mu_\infty) $ of $\R_I^n,$ where $\mbox{supp}(\mu_\infty) $ is the support of $\mu_\infty.$ Let us assume there exists a HK-integrable function $\Phi$ on $\R_I^n$ such that $\int_{\R_I^n} \Phi(x) dx =1 $ and $d \mu_\infty(x)= \Phi(x) dx ,$ we obtain for any element $f $ of $K{S^p}[\R_I^n]$ and for almost every element $ x $ of $\R_I^n$ as
	\begin{align*}
	\Phi_t * f(x) &= \int_{\R_I^n} f(x-y) \Phi_t(y) dy\\&= \int_{\R_I^n} f(x-ty) \Phi(y) dy \\&= M_{\mu_\infty}^{t} f(x).
	\end{align*}
\end{prop}
\begin{thm}
	Assume $\int_{\R_I^n} d \mu_\infty(x)=1,~ 1 \leq p \leq \infty$ and $ f \in W{S^{1,p}}[\R_I^n]$ that is $f $ and all its partial derivatives $ \frac{\partial f}{\partial x_i} ( 1 \leq j \leq m)$ are element of $K{S^p}[\R_I^n]$ then $$||f-M_{\mu_\infty}^{t} f|| \leq C[|\mu_\infty|[\R_I^n]]^{\frac{1}{q^{,}}}|||\nabla f|||\left( \sum_{r=1}^{\infty} \tau_r \int_{\R_I^n} \zeta_r(x)|y|^q d |\mu_\infty|(y)\right)^{\frac{1}{q}}t,~t \in (0, \infty),$$ where $|\nabla f|= \left( \sum\limits_{j=1}^{d}|\frac{\partial f}{\partial x_j}|^2\right)^{\frac{1}{2}}$ and $C$ is a positive real number not depending on $(\mu_\infty, f, t).$
\end{thm}
Using Theorem of Titchmarsh (see \cite{SI})  we can mention the following theorem.
\begin{thm} 
	If the divergence equation $\nabla . F= f~~(\mcE_f),$ where $(\mcE_f)$ has a solution $F= (\mcF_j)_{ i \leq j \leq d} $ in $KS^p[\R_I^n]$ and $f \in KS^p[\R_I^n]$ and there exists a positive real number $F$ such that the equation $(\mcE_f)$ has a solution in $WS^{k,p}[\R_I^n].$
\end{thm}
\begin{proof}
For proof of the result,  readers can see  \cite[Proposition 3.3]{SI}.
	\end{proof}
\section{Open Problem}
The Rellich-Kondrachov theorem gives us for $\Omega \subset \R_I^n$ is a bounded open set then $$W{S^{1,p}}[\R_I^n] \to W{S^{1, p-1}}[\R_I^n] $$ is compact. This will  help us to study the PDE in our space.\\
We are leaving this paper with an open problem: \\
Is the weak solution of PDE in Sobolev space also the strong solution in HK-Sobolev  Space?
\section{Conclusion} We conclude this paper with weakly differentiable in $L^p[\R_I^n]$ is strongly differentiable in $K{S^p}[\R_I^n]$ and this true when we replace $\R_I^\infty $ by $\R_I^n.$ 
 Weakly differentiable in $ W^{k,p}[\R_I^n]$ is strongly differentiable in $W{S^{k,p}}[\R_I^n]$  and  weakly convergence of $W^{k,p}[\R_I^n]$ and $W^{k,p}[\R_I^\infty]$ are strongly convergence in $W{S^{k,p}}[\R_I^n]$ and $W{S^{k,p}}[\R_I^\infty].$
  \section{Declaration}
\noindent  {\bf Funding:} Not Applicable, the research is not supported by any funding agency.\\
 {\bf Conflict of Interest/Competing interests:} The authors declare that there is no conflicts of interest.\\
 {\bf Availability of data and material:} The article does not contain any data for
 analysis.\\
 {\bf Code Availability:} Not Applicable.\\
 {\bf Author's Contributions:} All the authors have equal contribution for the preparation of the article.

\section*{Acknowledgment:} The authors would like to thank Prof. Tepper L. Gill for his valuable suggestions that improve the presentation of the paper. 
\bibliographystyle{amsalpha}

 \end{document}